\documentclass[12pt, reqno]{amsart}
\usepackage{amsmath, amsthm, amscd, amsfonts, amssymb, graphicx, color}
\usepackage{mathtools}
\usepackage{microtype}
\usepackage{float}
\usepackage[table]{xcolor}
\usepackage{enumitem}
\usepackage{algorithm}
\usepackage{algorithmic}
\usepackage{listings}
\usepackage[bookmarksnumbered, colorlinks, plainpages]{hyperref}

\newtheorem{theorem}{Theorem}
\newtheorem{lemma}{Lemma}
\newtheorem{proposition}{Proposition}

\theoremstyle{definition}
\newtheorem{definition}{Definition}

\theoremstyle{remark}
\newtheorem{remark}{Remark}
\numberwithin{equation}{section}
\counterwithin{figure}{section}
\usepackage{enumitem}

\begin{document}
	\setcounter{page}{1}

	{\small }

	\centerline{}
	
	\centerline{}
	
	\title{Inverse Scattering by Diffracted Waves}
\author[G. Bao]{Gang Bao}
\address{School of Mathematical Sciences, Zhejiang University, Hangzhou 310027, Zhejiang, China. }
\email{baog@zju.edu.cn}	
\author[X. Chen]{Xi Chen}
\address{Shanghai Center for Mathematical Sciences, Fudan University, Shanghai 200438, China; School of Mathematical Sciences, Fudan University, Shanghai 200433, China; 
	Center for Applied Mathematics, Fudan University, Shanghai 200433, China. }
\email{xi\_chen@fudan.edu.cn}
\author[S. Lu]{Shuai Lu}
\address{School of Mathematical Sciences, Fudan University, Shanghai 200433, China. }
\email{slu@fudan.edu.cn}
\author[K. Xiong]{Kuangmiao Xiong}
\address{Shanghai Center for Mathematical Sciences, Fudan University, Shanghai 200438, China. }
\email{24110840017@m.fudan.edu.cn}
	
	\begin{abstract}
In addition to reflection and refraction, another form of wave deviation is defined as diffraction. Notably, when incident waves strike a corner, diffracted waves emanate from the corner tip and propagate omnidirectionally. This paper proposes a novel framework for detecting rigid cornered obstacles using measured diffracted wave data. The framework first transforms the underlying initial-boundary value problems into initial value problems on conic manifolds via the method of images. Subsequently, the retrieval of obstacle information is achieved through Cheeger--Taylor functional calculus and microlocal analysis on conic manifolds. Specifically, we prove that for a given pulse, measurements of the resulting diffracted waves captured by a curve receiver uniquely determine both the location and shape of the visible portion of a polygonal obstacle. The proof is constructive, explicitly formulating the corresponding recovery scheme. This methodology offers two key advantages: first, the size and placement of the receiver can be arbitrary; second, the inversion only requires measurements of diffracted waves and obviates the need to solve wave equations within the cornered domain as in conventional methods.
	\end{abstract} 
	
	\maketitle
	
	
 	\section{Introduction}\label{sec1}

	 \subsection{Overview}
	 
The rigorous mathematical formulation of diffraction problems dates back to Sommerfeld \cite{So96}, who, in 1896, analyzed edge diffraction in the plane. In 1958, Friedlander \cite{Fr58} derived explicit formulae for the diffraction of plane and cylindrical waves by wedges in $\mathbb{R}^3$—a model equivalent to corner diffraction in $\mathbb{R}^2$. Later, Keller \cite{Ke62} developed the Geometrical Theory of Diffraction (GTD), which applies to a wider class of diffraction models.
	 \begin{figure}
	 	\centering
	 	\includegraphics[width=\textwidth]{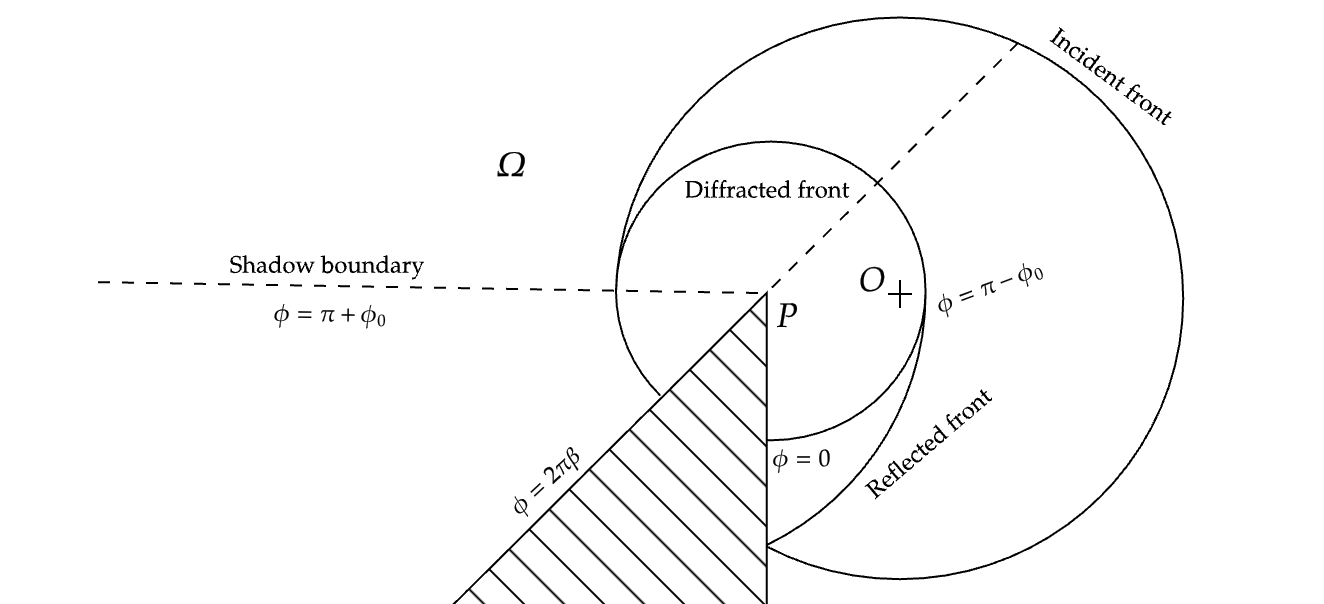}
	 	\caption{Diffraction of a spherical pulse by a corner}
	 	\label{figure0}
	 \end{figure}

	 Figure \ref{figure0}, taken from Friedlander \cite[Figure 5.2.(a)]{Fr58}, depicts the diffraction of a spherical pulse by a corner in 
	$\mathbb{R}^2$. In this model, a spherical wave originates from a point. The incident waves are those that do not impinge on the obstacle. When incident waves meet the obstacle beyond the tip, they get reflected by the boundary. When incident waves strike the very tip of the corner, diffraction takes place. The waves then split into diffracted waves and reflected waves. The diffracted waves are regarded as spherical waves emanating from the tip. This phenomenon lies outside the realm of classical geometric optics.

This example demonstrates that diffracted waves propagate omnidirectionally from the tip, regardless of the incident angles—thereby yielding more scattering data than reflected waves. By exploiting this phenomenon, we can reconstruct the shape and location of a corner using far fewer measurements than conventional active measurement techniques, which rely on reflected waves (e.g., the Dirichlet-to-Neumann map, the source-to-solution map, and far-field scattering data).

In this paper, we establish  that the location and visible shape of a fixed rigid polygonal obstacle can be detected using a single pulse and measurements of the resulting diffracted waves. Specifically, this only requires measuring the diffracted wave at a single time instant, with the receiver being arbitrary in both location and size.

Suppose $K \subset \mathbb{R}^2$ is a fixed and rigid polygonal obstacle   and $O \in \mathbb{R}^2$ is a point away from $K$. The rigidity of the underlying obstacle means that the wave velocity component normal to the obstacle vanishes on its boundary. Consequently, the boundary condition is of Neumann type. The model of concern is the initial boundary value problem in the exterior domain  $\Omega :=\mathbb{R}^2\setminus K$,
\begin{equation}\label{eq_1}
	\left\{ \begin{aligned}
		(\partial_{tt}-\Delta_{\Omega})u_{\Omega}(t,x,O)  &= 0&&  (t,x) \in \mathbb{R}\times \Omega,\\
		u_{\Omega}(0,x,O)  &=  0 &&  x\in \Omega, \\
		\partial_t u_{\Omega}(0,x,O)  &=  \delta_O  &&  x\in \Omega, \\
		\frac{\partial}{\partial v}u_{\Omega}(t,x,O)  &=  0 &&  (t,x)\in \mathbb{R}\times \partial \Omega, \\
	\end{aligned} \right.
\end{equation}
where $\Delta_{\Omega}$ is the Laplacian with Neumann boundary conditions and $v$ is the unit outward vector normal to $\partial \Omega$.

\subsection{Corner identification}
It is prudent to begin with the simpler model where the obstacle $K$ corresponds to a corner, including the concave corner case, as illustrated in Figure \ref{figure0}.

\begin{definition}
	A domain $K \subset \mathbb{R}^2$ is called a corner with tip $P$, if $\partial K$ consists of two straight rays meeting at $P$, with distinct one-sided tangent directions $v_1$ and $v_2$. Let the angle between $v_1$ and $v_2$ be denoted by $2\pi(1-\beta)$, as illustrated in Figure \ref{figure1}. Then:
	\begin{itemize}
		\item $K$ is called convex if $\beta \in (1/2, 1)$, meaning the angle $2\pi(1-\beta) \in (0, \pi)$.
		\item $K$ is called concave if $\beta \in (0, 1/2)$, meaning the angle $2\pi(1-\beta) \in (\pi, 2\pi)$.
		\item $K$ is called non-diffractive if $\beta \in \left\{N^{-1} : N \in \mathbb{Z}^+\right\}$. Otherwise, $K$ is called diffractive.
	\end{itemize}
\end{definition}

\begin{remark}Diffractive corners refer to those geometric singularities that induce wave diffraction, whereas non-diffractive corners do not give rise to such a phenomenon. As a matter of fact, the derivation presented in Remark \ref{concave corner which diffraction vanishes} indicates that the diffracted waves emanating from a corner will vanish when the corner is non-diffractive.

\end{remark}

The following theorem demonstrates that all diffractive corners are detectable, and their geometric shape and spatial location can be uniquely reconstructed.

\begin{figure}
	\centering
	\includegraphics[width=\textwidth]{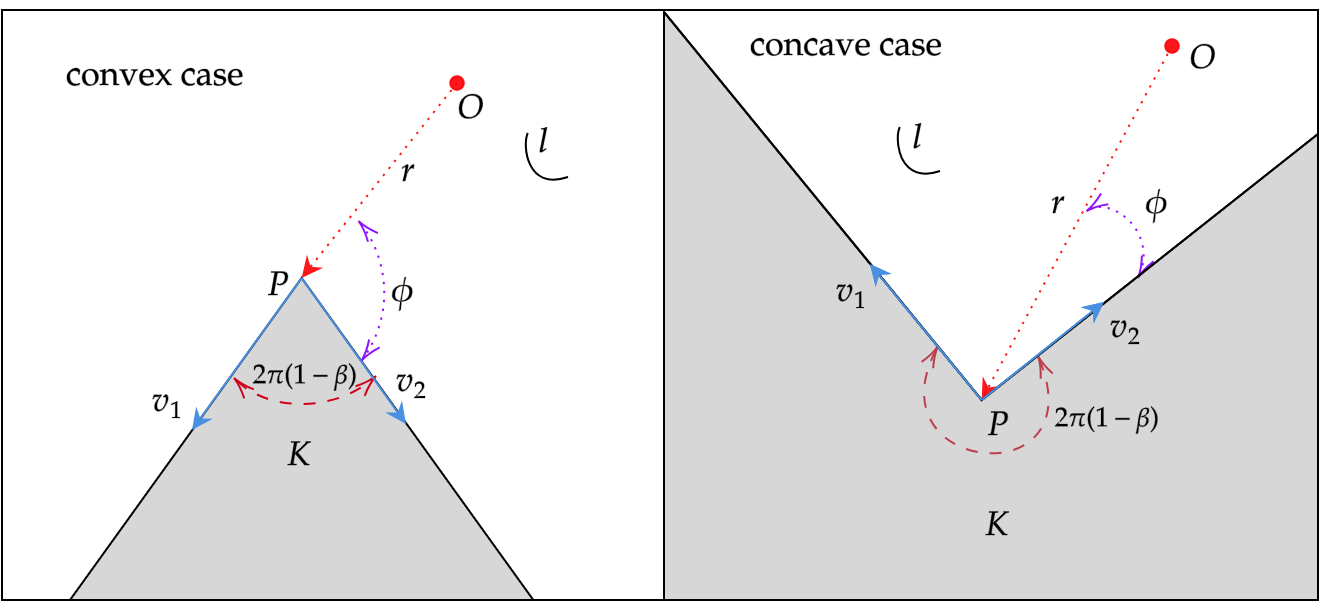}
	\caption{Determining the location and shape of a corner in $\mathbb{R}^2$}
	\label{figure1}
\end{figure}

\begin{theorem}\label{1.1} 
	Suppose $K \subset \mathbb{R}^2$ in \eqref{eq_1} is a diffractive corner with tip $P$. Let $O$ be a  point and $l$ a curved line segment in the exterior domain $\Omega$ of $K$, where $O$ and $l$ do not lie on the two branches of a hyperbola. Denote  by $r$ the distance from $O$ to $P$, and by $\phi$ the angle between $O$ and one edge of $K$. (See Figure~\ref{figure1}) Then the mapping
	\begin{equation}\label{eqn : observation at l}
		(P,r,\phi,\beta)\longrightarrow \left\{ u_{\mathcal{D},P}^{\mathrm{prin}}(t(x),x,O)  : x\in l \right\}
	\end{equation}
	is injective, where $u_{\mathcal{D},P}^{\mathrm{prin}}(t, x, O)$, given in \eqref{asy_0},  is the principal part of the wave $u_{\Omega}$ in \eqref{eq_1}, at $l$ with source $\delta_O$, and $t(x)$ is the arrival time of $u_{\Omega}$ at $x \in l$, i.e.  $$t(x):=r+r(x),$$
	where $r(x):=|x-P|$.
\end{theorem}

\begin{remark}
	If $O$ and $l$ lie on the two branches of a hyperbola, there exist $P_1$ and $P_2$ such that $$|O-P_1|+|P_1-x|=|O-P_2|+|P_2-x|,\quad x\in l.$$ Thus, we cannot distinguish between $P_1$ and $P_2$ based merely on the arrival times of $u_{\Omega}$.
\end{remark}

\begin{remark}
	Indeed, for each fixed $x$, we will show in Remark~\ref{another discription of T} that the arrival time of waves is
	\begin{equation}\label{eq: another discription of T}
		t(x) = \max \left\{ t>0 : x \in \mathrm{sing\,supp}(u_\Omega(t)) \right\}.
	\end{equation}
	This, in turn, leads to the definition of the $i$-th arrival time $t_i(x)$ of the waves in Theorem~\ref{1.2}.
\end{remark}

Unlike conventional inverse scattering methods, the proof of Theorem \ref{1.1} does not analyse the measurements of the exterior domain problem \eqref{eq_1}. 	  Instead,  the method of images is employed to convert the initial boundary value problem \eqref{eq_1} in $\Omega$ into a Cauchy problem of the wave equation on a cone. Given the radial symmetry of the domain $\Omega$, image sources can be introduced to replicate the effect of the boundary conditions, thereby reducing the initial boundary value problem to an initial value problem. In Theorem \ref{1.1}, we adopt polar coordinates in $\mathbb{R}^2$ with origin $P$ such that $\Omega=\mathbb{R}_+\times [0,2\pi\beta]$. Note that $\Omega$ is contained in the flat cone $C(I_{4\pi\beta})$, 
$$C(I_{4\pi\beta}):=\mathbb{R}_+\times I_{4\pi\beta}, \quad \mbox{with $I_{4\pi\beta}:=\mathbb{R}/4\pi\beta\mathbb{Z}$}.$$ The latter
 is furnished with the induced metric
\begin{equation}
	g_{C(I_{4\pi\beta})}=dr\otimes dr+r^2d\theta\otimes d\theta.
\end{equation}
If $\tilde{u}$ solves
\begin{equation}\label{conical}
	\left\{ \begin{aligned}
		(\partial_{tt}-\Delta_{C(I_{4\pi\beta})})\tilde{u}(t,x,O)  &= 0&&  (t,x) \in \mathbb{R}\times C(I_{4\pi\beta}),\\
		\tilde{u}(0,x,O)  &=  0 &&  x\in  C(I_{4\pi\beta}) ,\\
		\partial_t \tilde{u}(0,x,O)  &=    \delta_O+ \delta_{-O}  &&  x\in  C(I_{4\pi\beta}), \\
	\end{aligned} \right.
\end{equation}
where $O=(r,\phi)$ and $-O=(r,-\phi)$ such that $\delta_O+ \delta_{-O}$ is an even extension of $\delta_O$ to $C(I_{4\pi\beta})$, then the restriction of $\tilde{u}$ to $\Omega$ is the solution to \eqref{eq_1}. The readers are referred to \cite[Ch 3.7]{T_PDE_1} for more details on the method of images.

 Cheeger and Taylor \cite{Cheeger-PNAS-79, CT82a, CT82b} developed the functional calculus on cones, which enables us to compute the wave kernel and elucidate the diffraction of waves. In particular, the principal terms of diffracted waves on $\Omega$ explicitly read
 \begin{multline}\label{asy_0}
  u_{\mathcal{D},P}^{\mathrm{prin}}(t, x, O)= \\ -\frac{1}{2}(r(x)r)^{-1/2}H(t-r-r(x))\left[\sin\left(\pi\sqrt{\Delta_I}\right)\delta_{\phi}+\sin\left(\pi\sqrt{\Delta_I}\right)\delta_{-\phi}\right], 
 \end{multline}
 where $H(\cdot)$ is the Heaviside step function, $r(x)=|x-P|$, 
 and $\Delta_I$ is the Laplacian over $I_{4\pi\beta}$.  See Section \ref{sec3} for details.
It is the measurement \eqref{eqn : observation at l} of \eqref{asy_0} at the arrival times of the receiver that uniquely determines the location of the tip and the angle of the corner.

 \subsection{Polygon identification}

The diffraction framework in Theorem \ref{1.1} enables us  to detect the location and shape of a  compact polygon $K$ in $\mathbb{R}^2$ with vertices $\{P_i : i = 1, \dots, M\}$. For brevity of statement, we only address convex polygons. 
 
Since only the corners of visible tips are identifiable, we need to exclude the invisible tips.

\begin{definition} A vertex $P_i$ of $K$ is called visible, with respect to $O$ and $l$, if the line segments connecting $P_i$ to $O$ and to every point on $l$ lie entirely within $\Omega$. (See Figure~\ref{figure2.1})
\end{definition} 

\begin{figure}
	\centering
	\includegraphics[width=\textwidth]{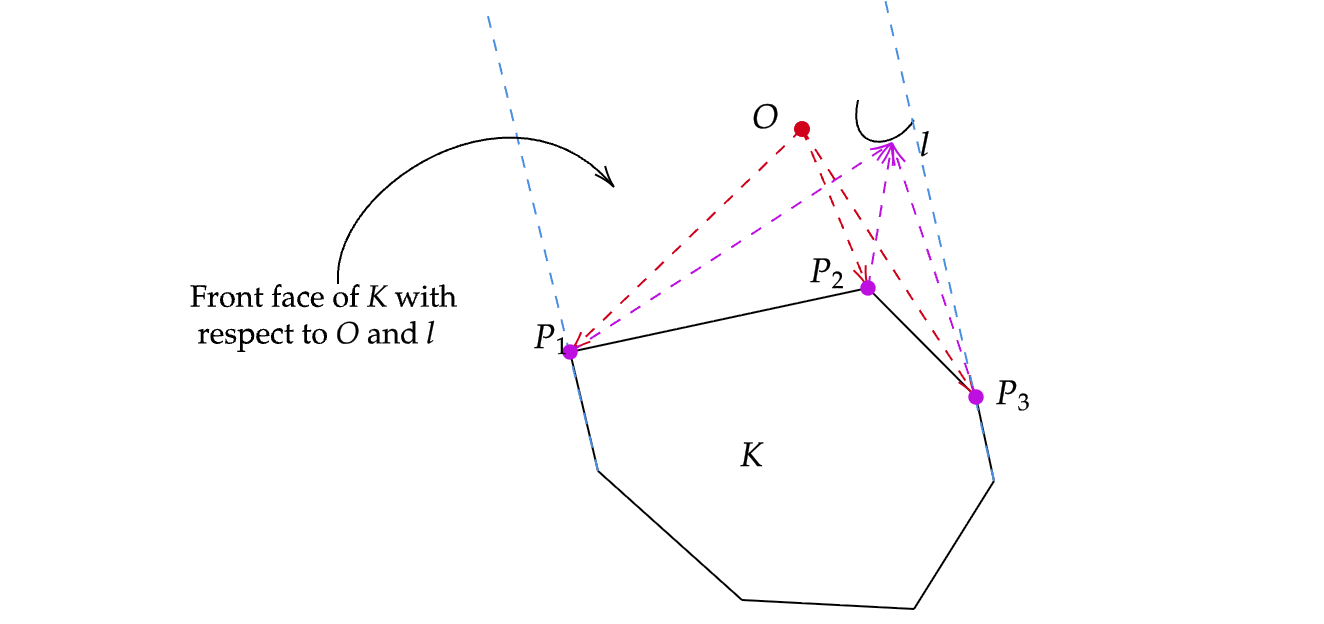}
	\caption{Visible vertices of $K$ with respect to $O$ and $l$}
	\label{figure2.1}
\end{figure}

Every visible vertex $P_i$ must be struck by any spherical waves originating from $O$. Consequently, waves diffract at each $P_i$. Then  the diffracted waves $u_{\mathcal{D},P_i}$ are produced and then observed at time $t_i(x)$ at point $x \in l$.  

However,   reconstructing $K$ is by no means a simple iteration of Theorem \ref{1.1} from $P_1$ to $P_m$.  The reasons are threefold. First of all, a diffracted wave $u_{\mathcal{D},P_i}$, emanating from any visible vertex $P_i$, propagates along all directions, thereby causing a secondary diffraction at the adjacent vertex $P_j$. While this wave is observed at the receiver, it also generates a new diffracted wave $u_{\mathcal{D},P_i,P_j}$ at $P_j$, which may corrupt the measurements. 
Second, reflected waves from other edges also contaminate the observations at $l$.
Furthermore, the arrival order of  $u_{\mathcal{D},P_i}$ from each $P_i$ vary within $l$ as the travel time of the diffracted wave $|O-P_i|+|x - P_i|$ depends on both $i = 1, \cdots, N$ and $x \in l$. Compared with the model case in Theorem \ref{1.1}, this creates a challenge in separating the measurements of the diffracted waves from distinct vertices.

To resolve these difficulties above, we use different techniques. For the first two issues, one can filter out both reflected waves and secondary diffracted waves by capturing only those waves that fall within a specified magnitude range. This is because incident and reflected waves are stronger than all diffracted waves, with primary diffracted waves being stronger than secondary ones. To separate the diffracted waves from distinct vertices, localization along  $l$
is required to capture the earliest arriving wave within a small neighbourhood on $l$.

Once this issue addressed, one can detect both the location and corner angle of any visible tip of $K$.

We begin by providing a rigorous definition of the observational data used.

\begin{definition}
	Let $\Omega\subset \mathbb{R}^n$ be an open set and $x_0\in\Omega$.  Define an equivalence relation on $\mathcal{D}'(\Omega)$: 
	$$u \sim_{x_0} v \quad\mbox{if and only if}\quad x_0\notin \mathrm{supp}(u-v).$$ The equivalence class of $u$ under this relation is called the localization of $u$ at $x_0$, and denoted by $[u]_{x_0}$.
\end{definition}

	We use this type of data to determine the location and shape of polygonal obstacles.
	\begin{theorem}\label{1.2}
		Let $K$ in \eqref{eq_1} be a compact convex polygon in $\mathbb{R}^2$, $\mathrm{Vert}(K)$ the set of vertices of $K$, and $O\in \Omega:=\mathbb{R}^2\setminus K$. Suppose that there is no grazing ray emanating from the point $O$, i.e,
		\begin{equation}\label{no-grazing condition}
			\mbox{the line through $O$ is not tangent to $\partial K$.}
		\end{equation}
		Furthermore, suppose $l\subset\Omega$ is a curved line segment that is not formed by the piecewise composition of straight line segments and hyperbolic arcs alone. Let $\{P_{j}\}_{j=1,...,N}$ denote the collection of visible vertices of $K$ with respect to $O$ and $l$. Denote $(P_{j},r_{j},\phi_{j},\beta_{j})$ as in Theorem \ref{1.1}. The mapping
		\begin{multline}\label{eqn : poly observation at l}
		D: \left\{\left(P_{j},r_{j},\phi_{j},\beta_{j}\right) : {j=1,...,N}\right\} \\ \longrightarrow  \left\{\left[u_{\Omega}(t_i(x_0),x,O)\right]_{x_0} : x_0\in l,\, i\in \mathbb{Z}_+\right\}
		\end{multline}
		is injective, where  $t_i(x)$ is the $i$-th arrival time of the waves in \eqref{eq_1} at $x$, defined by
		\[
		t_{i+1}(x) := \min\left\{ t > t_i(x) : x \in \mathrm{sing\,supp}(u_\Omega(t)) \right\},
		\]
		with the first arrival time
		\[
		t_1(x) := \min\left\{t > |x - O| : x \in \mathrm{sing\,supp}(u_\Omega(t)) \right\}.
		\]
		
\end{theorem}

	\begin{remark}
		As stated in \cite{CT82b,MW04}, if a grazing ray occurs, the waves along this ray have the same Sobolev regularity with incident waves. Consequently, the secondary diffracted waves  can not be filtered out by the quotient mapping \eqref{quotient mapping}.
	\end{remark}	

 \begin{remark}
 	We shall see that, on the one hand, the locus of points at which the primary diffracted waves generated by $P_i$ and $P_j$ arrive exactly at the same time is a (branch of a) hyperbola, possibly degenerating to a straight line. Moreover, the points at which the diffracted front and the geometric front arrive at the same time lie on only finitely many straight lines passing through the tips. This is precisely the reason for imposing the restriction on the receiver $l$ stated in the theorem.
 \end{remark}
 
\subsection{Discussion}Time-domain inverse obstacle scattering problems have garnered many applications from diverse fields, including non-destructive testing, mineral prospecting, the characterization of acoustic properties in buildings, ultrasound imaging, radar systems, and sonar technology. Classical inverse scattering theory, such as the framework detailed in \cite{Colton-Kress}, usually reconstructs the geometric profile of the target obstacle via reflection tomography. As far as we are aware, no existing inversion results have been derived by relying exclusively on the data of diffracted waves.

For detecting polygonal obstacles, the proposed diffracted wave approach features two key advantages. First, as shown in \eqref{eqn : poly observation at l}, it requires only a single spherical wave from an arbitrary point source, an arbitrary line segment receiver, and observation of the first arrival time for each tip. Second, inversion avoids solving the boundary value problems of associated wave equations, relying solely on the oscillatory factors of diffracted waves derived from functional calculus on conic manifolds.
 
 The diffraction phenomenon on manifolds with singularities is a long-standing problem in analysis and geometry. The core question is to understand how geometric singularities affect wave propagation—a question not addressed by the classical theorem of singularity propagation in smooth media due to H\"ormander and Duistermaat \cite{Ho_FIO_1, Ho_FIO_2}.
 
 For product-type conic manifolds (those with radial symmetry), \cite{CT82a, CT82b} not only computed the explicit wave kernel using the methods of separation of variables and functional calculus but also established a theorem for the propagation of diffraction singularities. However, this methodology fails for non-product-type conic manifolds (those without radial symmetry). Wave propagation in diffraction scenarios has been investigated through various analytic and smooth microlocal analysis frameworks in more general settings, including the analytic category \cite{Leb97}, non-product-type conic manifolds \cite{MW04, GW-JDG-22}, manifolds with edges \cite{MVW08}, manifolds with corners \cite{V08}, manifolds with conormal singularities \cite{dHUV-2015},  and Lorentzian manifolds with conic singularities \cite{hintz2024localtheorywaveequations}.  Beyond wave diffraction in the time domain, \cite{BPS-CMP} established the absence of non-scattering wavenumbers for penetrable corner obstacles in the frequency domain.
 
The aforementioned research on forward diffraction problems lays a solid foundation for the rigorous mathematical understanding of diffraction phenomena. This paper, by contrast, focuses on the inverse problem of reconstructing the geometric parameters of obstacles from local measurements of diffracted waves. To the best of our knowledge, this work represents the first attempt to apply the geometric and microlocal analysis of diffracted waves to the study of inverse obstacle problems.  Through a refined analysis of diffracted waves, we establish that the location and visible shape of a fixed rigid polygonal obstacle can be uniquely determined using a single pulse and measurements of the resulting diffracted waves. Remarkably, the geometric and microlocal information of diffracted waves significantly reduces observational costs, as it requires measurements at only a single time instant, with the receiver being arbitrary in both its location and size.

	\subsection*{Structure of the paper}
Section \ref{sec2} reviews relevant geometry and analysis on conic manifolds. Section \ref{sec3} applies functional calculus to derive an explicit formula for waves near a corner in \(\mathbb{R}^2\), and proves unique parameter retrieval from the oscillatory component of diffracted waves (Theorem \ref{T4.1}). Theorems \ref{1.1} and \ref{1.2} are proven in Sections \ref{P1.1} and \ref{P1.2} respectively.

	\section{The geometry and analysis of conic manifolds}\label{sec2}
	\subsection{The geodesics on conic manifolds}\label{sec2.1}
	Let $(\Theta, g_{\Theta})$ be a complete  Riemannian manifold. Denote by $\gamma_{(\theta_0,\eta_0)}(t)$   the unique geodesic starting at $\theta_0 \in \Theta$ with initial velocity $\eta_0 \in T_{\theta_0}^\ast \Theta$, and for brevity write
	\begin{equation*}
		\left\{ \begin{aligned}
					\gamma_{(\theta_0,\eta_0)}(t)  &= \theta(t), && \gamma_{(\theta_0,\eta_0)}(0)   = \theta_0; \\
			\dot{\gamma}_{(\theta_0,\eta_0)}^\flat(t)  &= \eta(t),&&
			\dot{\gamma}_{(\theta_0,\eta_0)}^\flat(0)  = \eta_0. 
		\end{aligned} \right.
	\end{equation*}  
	The geodesic flow   $$\left\{\left(\gamma_{(\theta_0,\eta_0)}(t), \dot{\gamma}^\flat_{(\theta_0,\eta_0)}(t)\right) : t > 0\right\} \subset T^\ast \Theta $$ is the Hamiltonian flow $H_t$ with Hamiltonian $$H(\theta, \eta) := \frac{1}{2} g_\Theta^{ij}(\theta) \eta_i \eta_j, \quad \mbox{with $\left(g_\Theta^{ij}\right)^{-1} = g_\Theta$}.$$  Explicitly, the geodesic flow is given by 
	\begin{equation*}
		\left\{ \begin{aligned}
			\dot{\theta}^i  &=  g_\Theta^{ij}(\theta)\eta_j \\
			\dot{\eta}_k&= -\frac{1}{2} \partial_{\theta^k} g_\Theta^{ij}(\theta)\eta_i\eta_j.\\
		\end{aligned} \right.
	\end{equation*}
	
	Let $C(\Theta):=\mathbb{R}_+\times \Theta$ be a conic manifold, where $(\Theta,g_\Theta)$ is a compact Riemannian manifold without boundary, and $C(\Theta)$ is furnished with metric
	\begin{equation}\label{metric}
		g_{C(\Theta)}(r, \theta):=dr\otimes dr+r^2g_\Theta(\theta).
	\end{equation}
	For convenience, we denote the completed cone of $C(\Theta)$  	\begin{displaymath}
		C^*(\Theta):=\{(r,\theta)\in C(\Theta):r \geq 0\}
	\end{displaymath}as well as the truncated cone  
	\begin{displaymath}
		C_{r_1,r_2}(\Theta):=\left\{(r,\theta)\in C(\Theta): 0 \leq r_1<r<r_2\leq \infty\right\}.
	\end{displaymath}
	Consider the spherical cotangent bundle \[ S^* C(\Theta) := \left\{  (r,\theta,\tau,\eta)\in T^* C(\Theta) :    \tau^2 + r^{-2}g^{ij}_\Theta(\theta)\eta_i\eta_j \equiv 1 \right\}.\]    
	The geodesic flow in $S^\ast C(\Theta)$ is given by
	\begin{equation}\label{geo}
		\left\{ \begin{aligned}
			\dot{r}  &= \tau;  \\	
				\dot{\theta}^i  &=  r^{-2}g^{ij}_\Theta(\theta)\eta_j; \\
			\dot{\tau} &= r^{-3}g^{ij}_\Theta(\theta)\eta_i\eta_j ;\\
			\dot{\eta}_{k}&= -\frac{1}{2}r^{-2} \partial_{\theta^k} g^{ij}_\Theta(\theta)\eta_i\eta_j.\\
		\end{aligned} \right.
	\end{equation}
	Given the initial data \begin{equation*}
		(r(0), \theta(0), \tau(0), \eta(0)) = (r_0, \theta_0, \tau_0, \eta_0) \in S^*C(\Theta)
	\end{equation*} 
	for the geodesic equation \eqref{geo}, we denote by $\tilde{\gamma}_{(r_0, \theta_0, \tau_0, \eta_0)}(t)$  the corresponding geodesic on $C(\Theta)$. The equations of $r$ and $\tau$ can be transformed into
	\begin{equation}\label{geo_r}
		\left\{ \begin{aligned}
			\dot{r}  &= \tau; \\
			\dot{\tau} &= r^{-1}(1-\tau^2). \\
		\end{aligned} \right.
	\end{equation} 
 The solution to \eqref{geo_r} is
	\begin{equation}\label{eq 1}
	\left\{ \begin{aligned}  r(t)&=\sqrt{\left(t + \tau_0 r_0\right)^2 + C_0^2  },\\ \tau(t)&=\frac{t + \tau_0 r_0}{\sqrt{\left(t + \tau_0 r_0\right)^2 + C_0^2 }},   \end{aligned} \right.
	\end{equation}
with $C_0 :=r_0\sqrt{1-\tau^2_0}$.

The constant $C_0$ in \eqref{eq 1} indicates whether the underlying geodesic $\tilde{\gamma}_{(r_0, \theta_0, \tau_0, \eta_0)}(t) $ passes through the cone tip $P$ of $C(\Theta)$.

When $C_0 = 0$ (i.e. $\tau_0 = \pm 1$), it follows from \eqref{eq 1} and \eqref{geo_r} that
	\begin{equation*}
		\tau(t) = 
		\begin{cases}
			1 & \text{if } t > -r_0 \tau_0 \\
			-1 & \text{if } t < -r_0 \tau_0
		\end{cases}S.
	\end{equation*}
Since geodesics travel at constant speed, $\tilde{\gamma}_{(r_0, \theta_0, \tau_0, \eta_0)}(t)$ leaves $S^\ast C(\Theta)$ invariant, i.e. $$ r^{-2} g_\Theta^{ij}(\theta) \eta_i\eta_j =  {1 - \tau^2} \equiv 0$$ over the geodesic flow. Since $g_\Theta$ is positive definite, $\eta \equiv 0$. This means that the geodesic encoded in \eqref{geo} is simply $\tilde{\gamma}_{(r_0, \theta_0, \tau_0, \eta_0)}(t) = (\pm t + r_0, \theta_0)$, passing through $r=0$ at time $t=\mp r_0$.
	
Now	we assume that $C_0\neq0$. By \eqref{eq 1}, $r(t)$ is bounded from below by $C_0 > 0$ and the underlying geodesic never goes to the cone tip. Using the shorthand notations $\tilde{\eta}:=C^{-1}_0\eta$ and $s(t):=\arctan\left({(t+r_0\tau_0)}/{C_0}\right)$, we transform the second and fourth equations of \eqref{geo}  into the unit speed geodesic equations  on $\Theta$,
	\begin{equation*}
		\left\{ \begin{aligned}
			\frac{d}{ds}\theta^i  &= g^{ij}_\Theta\tilde{\eta}_j; \\
			\frac{d}{ds}\tilde{\eta}_k&= -\frac{1}{2} \partial_{\theta^k} g^{ij}_\Theta\tilde{\eta}_i\tilde{\eta}_j. \\
		\end{aligned} \right.
	\end{equation*}
Consequently,  the projection of a geodesic with parameter $t$ from $C(\Theta)$ onto $\Theta$ corresponds exactly to a geodesic with parameter $s$ on $\Theta$. Note that
	$$
	\lim_{t \to \pm\infty} s(t) = \pm\frac{\pi}{2}  .
	$$
This implies that the maximum distance $d_{\max}$ on $\Theta$ is 
\begin{equation}\label{eqn: maximal traval time on Theta}
	d_{\max} = s(+\infty) - s(-\infty) = \pi.
\end{equation}
This is illustrated in Figure~\ref{figure_1.1} from \cite[Figure 2.]{MW01}.

	\begin{figure}
		\centering
		\includegraphics[width=\textwidth]{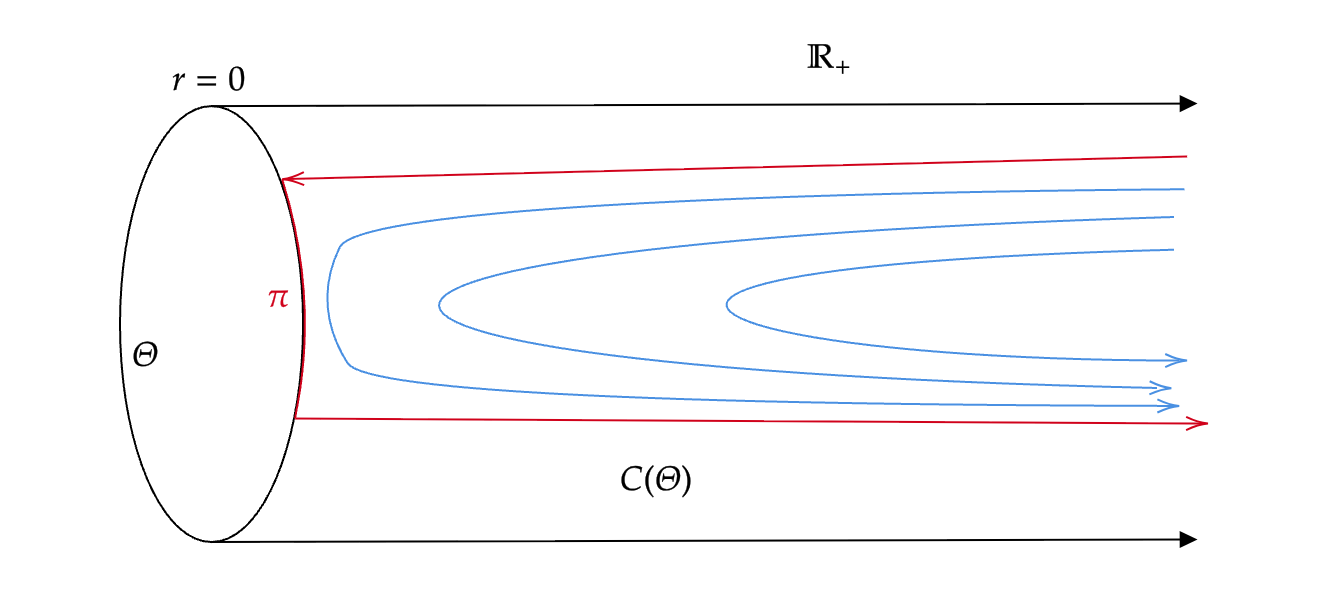}
		\caption{Maximum distance on $\Theta$}
		\label{figure_1.1}
	\end{figure}

	The triangle inequality on $C(\Theta)$,
		\begin{equation}\label{eqn: traval time}
	 	d(O, x)  \leq d(O,P)+d(x,P),
		\end{equation}with the distance function  $d(\cdot, \cdot)$ on $C(\Theta)$,
 suggests that, at any observation point $x$, the diffracted wave never arrives earlier than the incident waves whose paths do not intersect $P$.

	\subsection{The wave kernel on conic manifolds}
	In this section only, we use $(r_1,\theta_1)$ and $(r_2,\theta_2)$ to denote the left and right variables, respectively, of the Schwartz kernel of an operator.
	
	Let $(\Theta,g_\Theta)$ be a compact Riemannian $n$-manifold  without boundary, $\Delta_{\Theta}$ the Laplacian on $\Theta$, and $\Delta_{C(\Theta)}$ the Friedrichs extension of the Laplacian on $C(\Theta)$. 
	Away from the cone tip, $\Delta_{C(\Theta)}$ is given explicitly by
	\begin{equation}\label{lap}
		\Delta_{C(\Theta)}=-\frac{\partial^2}{\partial r^2}-\frac{n}{r}\frac{\partial}{\partial r}+\frac{1}{r^2}\Delta_\Theta.
	\end{equation}

Following \cite{CT82a}, for $k \in \mathbb{R}$, we denote by $\mathcal{D}^k$ the domain  of $\Delta_{C(\Theta)}^k$, which is defined by the norm 
	$$\left\|g\right\|_{\mathcal{D}^k} := \left\|(\mathrm{Id} + \Delta_{C(\Theta)})^k g\right\|_{L^2},$$ and  $\mathcal{D}^\infty:=\cap_{k\in\mathbb{Z}}\mathcal{D}^k$. In addition, we say \( f \in \mathcal{D}^k_{\mathrm{loc}}(\mathcal{U}) \) for  \(\mathcal{U} \subset C(\Theta)\) if there exists \( g \in \mathcal{D}^k(C(\Theta)) \) with \( f \equiv g \) on \(\mathcal{U}\).
	
For $i \in \mathbb{Z}_+$, suppose $\mu_i$ is the $i$-th eigenvalue of $\Delta_\Theta$, $\varphi_i$ is an eigenfunction of $\Delta_\Theta$ associated with $\mu_i$, and $\nu_i :=  \left(\mu_i + \left( n-1 \right)^2/4\right)^{1/2}$. By \cite[(0.5)-(0.7)]{CT82a}, the Schwartz kernel $K_{f(\lambda)}$ of operator $f(\sqrt{\Delta_{C(\Theta)}})$, for a Borel function $f$ on $\mathbb{R}$, takes the form 
	\begin{equation}\label{eqn : functional calculus}
	K_{f(\lambda)}(r_1,\theta_1,r_2,\theta_2)=(r_1r_2)^{-\frac{n-1}{2}}\sum_{i=0}^{\infty}\tilde{K}_{f(\lambda)}(r_1,r_2,\nu_i)\varphi_i(r_1, \theta_1)\bar{\varphi}_i(r_2, \theta_2),
	\end{equation}
	where $(r_1, \theta_1), (r_2, \theta_2) \in C(\Theta)$,
	\begin{equation*}
		\tilde{K}_{f(\lambda)}(r_1,r_2,\nu):=\int_{0}^{\infty}f(\lambda)J_\nu (\lambda r_1)J_\nu (\lambda r_2)\lambda d\lambda,
	\end{equation*}
	and $J_\nu(x)$ is the Bessel function of first kind.

	 Applying \eqref{eqn : functional calculus} to function $\sin (t\lambda)/\lambda$ yields the wave kernel $$K_{\sin (t\lambda)/\lambda} := \mathrm{Ker} \frac{\sin \left(t\sqrt{\Delta_{C(\Theta)}}\right)}{\sqrt{\Delta_{C(\Theta)}}}.$$ To shed light on the structure of the wave kernel, we carry out a partition of $C(\Theta)$	in line with wave propagation.
	 
	 At time $t$, the union of the wave fronts of incident and reflected waves emanating from $O$ is called the geometric wave front with respect to $O$ and denoted by $\mathcal{G}_{t}^{O}$, whereas the wave front $\mathcal{D}_{t}^{O}$ of diffracted waves with respect to $O$ is given by
	 \begin{equation*}
	 	\mathcal{D}_{t}^{O} := \left\{x\in \Omega:d(Q,P)=t-d(O,P)\right\}.
	 \end{equation*}
	 In addition,  $\mathcal{D}_{t}^{O}$ and the closure of $\mathcal{G}_{t}^{O}$ have a non-empty intersection $\mathcal{F}_{t}^{O}$. \cite[p. 322]{CT82a} claims the exact form of $\mathcal{F}_{t}^{O}$. We give a short proof for completeness.

	 \begin{lemma}\label{intersection}
	 	Let $O = (r,\phi)$ and $\xi\in S^*_{\phi}\Theta$. Denote by $\pi_\Theta : C(\Theta) \to \Theta$ the canonical projection. For $t\geq r_2$, we have
	 	\begin{align*}
	 		&\mathcal{F}_{t}^{O} = \{(t - r,\theta) : \exists\,\xi\in S^*_{\phi}\Theta, \, \theta=\gamma_{(\phi,\xi)}(\pi)\in\Theta \},\\
	 		&\pi_\Theta(\mathcal{F}_{t}^{O})=\{\gamma_{(\phi,\xi)}(\pi)\in\Theta:\xi\in S^*_{\phi}\Theta\}.
	 	\end{align*}
	 	When $\dim \Theta = 1$, $|\pi_\Theta(\mathcal{F}_{t}^{O})|$ is finite and $\pi_\Theta(\mathcal{F}_{t}^{O})$ is independent of $t$. Consequently, $$\mathcal{F}_{O}:=\bigcup_{t>r_2}\mathcal{F}_{t}^{O}=\mathbb{R}_+\times\pi_{\Theta}(\mathcal{F}_{O}),$$ and $|\pi_\Theta(\mathcal{F}_{O})|$ is finite.
	 \end{lemma}
	 \begin{proof}
	 	Let $(r_1,\theta_1)\in \mathcal{F}_{t}^{O}$. Since $\Theta$ is boundaryless, the geometric wavefront is  comprised only of  incident waves.  Hence, we have
	 	\begin{multline}\label{discription of G_t^O}
	 		\mathcal{G}_{t}^{O}
	 		= \Bigl\{ \tilde{\gamma}_{(O,\tilde{\eta})}(t)
	 		:\tilde{\eta} = (\tau(0),\eta)\in S^*_OC(\Theta),\ \mbox{i.e.}\\  \tau(0)^2+r^{-2}|\eta|_{\Theta}^2 =1,\ \mbox{and}\ \tau(0)\neq -1 \Bigr\},
	 	\end{multline}
	 	where $\tilde{\gamma}_{(O,\tilde{\eta})}(t)$ denotes the geodesic in $C(\Theta)$ with initial date $(O, \tilde{\eta})$. 
	 	By \eqref{eq 1}, we have $t =  -r \tau(0) + r_1 \tau(t)$ on $\mathcal{G}_{t}^{O}$. 
	 	Since $t = r + r_1$ with $r,r_1>0$ on $\mathcal{D}_{t}^{O}$, it follows that $\tau(0) = -1$ and $\tau(t)=1$ on $\mathcal{G}_{t}^{O}\cap\mathcal{D}_{t}^{O}$ which contradicts \eqref{discription of G_t^O}. Consequently,
	 	\begin{equation}\label{inclusions}
	 		\mathcal{D}_{t}^{O} \cap \mathcal{G}_{t}^{O} = \emptyset,
	 	\qquad
	 	\mathcal{F}_{t}^{O} \subset \bar{\mathcal{G}}_{t}^{O} \setminus \mathcal{G}_{t}^{O}.
	 	\end{equation}
	 	
	 	Rewriting $(O,\tilde{\eta})$ in the form $(r,\phi, \pm\sqrt{1-\epsilon^{2}},  \epsilon r\eta)$, we obtain 
	 	\begin{multline*}
	 		\mathcal{G}_{t}^{O}
	 		= \Bigl\{ \tilde{\gamma}_{(O,\tilde{\eta}^-_\epsilon)}(t)
	 		:\tilde{\eta}^-_\epsilon = (-\sqrt{1-\epsilon^{2}},\epsilon r\eta)\in S^*_OC(\Theta),\ |\eta|_\Theta=1,\  \epsilon\in(0,1] \Bigr\}\\
	 		\bigcup  \Bigl\{ \tilde{\gamma}_{(O,\tilde{\eta}_\epsilon^+)}(t)
	 		:\tilde{\eta}_\epsilon^+ = (\sqrt{1-\epsilon^{2}},\epsilon r\eta)\in S^*_OC(\Theta),\ |\eta|_\Theta=1,\  \epsilon\in [0,1] \Bigr\}.
	 	\end{multline*}
	 	Let $(r^-_\epsilon(t),\theta^-_\epsilon(t)):=\tilde{\gamma}_{(O,\tilde{\eta}^-_\epsilon)}(t)$. Then, we use \eqref{eq 1} to derive
	 	\[
	 	r_\epsilon^-(t)
	 	= \sqrt{\bigl(t - \sqrt{1-\epsilon^{2}}\, r \bigr)^2 + \epsilon^{2} r^{2}}
	 	\;\longrightarrow\; r_1 \quad\mbox{as}\quad \epsilon\to 0,
	 	\]
	 	and we have $r+r_1=t$. Hence $\bar{\mathcal{G}}_{t}^{O} \setminus \mathcal{G}_{t}^{O} \subset \mathcal{D}_{t}^{O}.$ Together with \eqref{inclusions}, we have $\mathcal{F}_{t}^{O}=\bar{\mathcal{G}}_{t}^{O} \setminus \mathcal{G}_{t}^{O}.$
	 	 
	 	By \eqref{eqn: maximal traval time on Theta} and \eqref{inclusions}, we obtain 
	 	\[
	 	\mathcal{F}_{t}^{O}=
	 	\bar{\mathcal{G}}_{t}^{O} \setminus \mathcal{G}_{t}^{O} = \{(t - r,\theta) : \exists\,\xi\in S^*_{\phi}\Theta, \, \theta=\gamma_{(\phi,\xi)}(\pi)\in\Theta \}.
	 	\]
	 	Projecting to $\Theta$ gives $\pi_\Theta(\mathcal{F}_{t}^{O})=\{\gamma_{(\phi,\xi)}(\pi)\in\Theta:\xi\in S^*_{\phi}\Theta\}$. This implies that $\pi_\Theta(\mathcal{F}_{t}^{O})$ is independent of $t$. When $\dim \Theta = 1$, we have $S^*_{\phi}\Theta = \{-1,1\}$, and therefore
	 	\[
	 	|\pi_{\Theta}(\mathcal{F}_{t}^{O})|
	 	\le |S^*_{\phi}\Theta| = 2.
	 	\]
	 \end{proof}
	 
As in \cite{Cheeger-PNAS-79, CT82a, CT82b},	we fix $O \in C(\Theta)$ and $t > 0$, and divide $C(\Theta)$ by $\mathcal{D}_{t}^{O}$ and $\mathcal{G}_{t}^{O}$ as follows :
	 \begin{equation*}\label{eqn : partition of C(Theta)} 
	 	\left\{
	 	\begin{aligned}
	 	&\mbox{Region I:} &&0<t<d(Q,O),\\
	 	&\mbox{Region II:} &&d(Q,O)<t<d(Q,P)+d(O,P),\\
	 	&\mbox{Region III:} &&t>d(Q,P)+d(O,P),
	 	\end{aligned}\right.
	 \end{equation*} where  $O = (r, \phi)$, $Q=(r_1, \theta_1)$, and $P$ is the cone tip.
	 	 As shown in Figure \ref{figure0.1} for the exterior domain of a wedge in $\mathbb{R}^2$, Region I is the area where no wave front has appeared before time $t$; Region II is the area through which the geometric wave front has passed but the diffracted wave front has not yet reached; Region III is the area surrounded by the diffracted wave front.

	 \begin{figure}
	 	\centering
	 	\includegraphics[width=\textwidth]{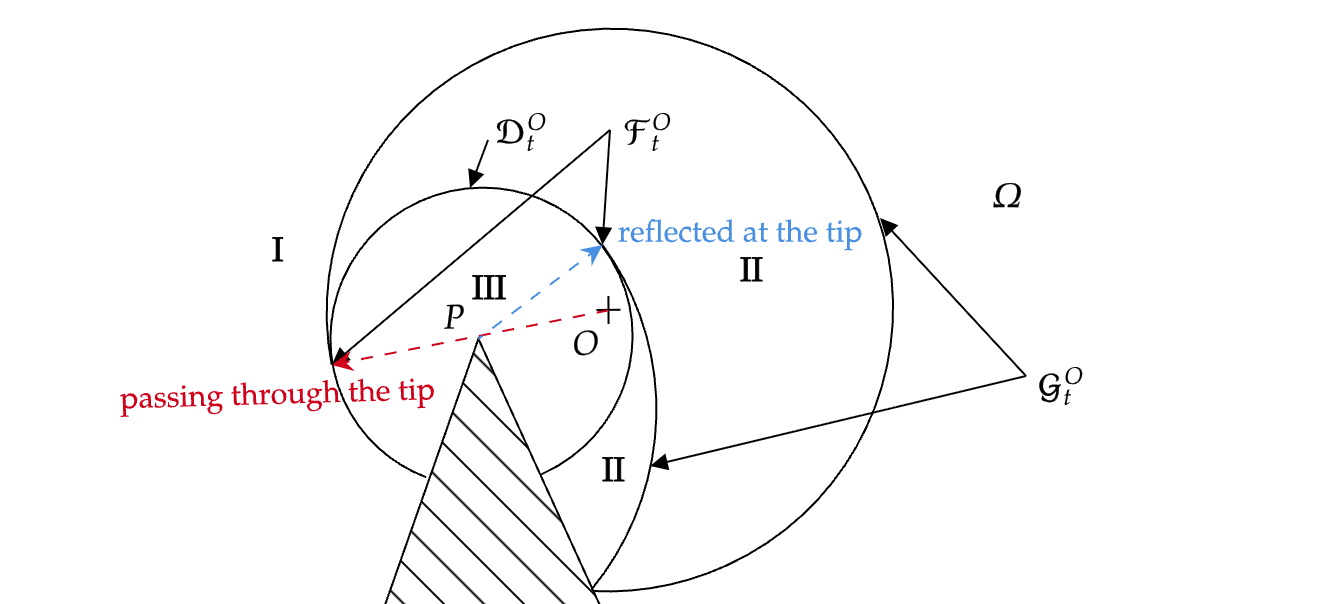}
	 	\caption{Space partition of $\Omega$}
	 	\label{figure0.1}
	 \end{figure}

	By \cite[(3.26), (3.27), (3.29)]{CT82a}, the wave kernel  $K_{\sin t\lambda/\lambda}$ takes the following specific forms:
	\begin{itemize}
		\item Region I ($0<t<|r_1-r_2|$):
		\begin{equation*}
			K_{\sin t\lambda/\lambda}(r_1,\theta_1,r_2,\theta_2)\equiv 0,
		\end{equation*}
		\item Region II ($|r_1-r_2|<t<r_1+r_2$): 
		\begin{align}\notag
		\lefteqn{K_{\sin t\lambda/\lambda}(r_1,\theta_1,r_2,\theta_2)} \\ &=	\frac{1}{\pi}(r_1r_2)^{-{(n-1)}/{2}}\int_{0}^{\beta_1}\frac{\cos (\nu s) ds}{(t^2-(r_1^2+r_2^2-2r_1r_2\cos s))^{1/2}}, \label{II_eq}
		\end{align}
		\item Region III ($t>r_1+r_2$):
		\begin{align}\label{III_eq}
	\lefteqn{K_{\sin t\lambda/\lambda}(r_1,\theta_1,r_2,\theta_2)} \nonumber\\
		 	&=\frac{1}{\pi}(r_1r_2)^{-{(n-1)}/{2}}\int_{0}^{\pi}\frac{\cos (\nu s)ds}{(t^2-(r_1^2+r_2^2-2r_1r_2\cos s))^{1/2}}\nonumber\\
			&\quad - \frac{1}{\pi}(r_1r_2)^{-{(n-1)}/{2}}\sin \pi\nu \int_{0}^{\beta_2}\frac{e^{-\nu s}ds}{(t^2-(r_1^2+r_2^2+2r_1r_2\cosh s))^{1/2}},
		\end{align}
	\end{itemize}
	where we used the shorthand notations  
	\begin{align*}\nu&:= \left(\Delta_{\Theta}+\frac{\left(n-1\right)^2}{4}\right)^{1/2}\\
		\beta_1&:=\cos^{-1}\left(\frac{r_1^2+r_2^2-t^2}{2r_1r_2}\right)\\
		\beta_2&:=\cosh^{-1}\left(\frac{t^2-r_1^2-r_2^2}{2r_1r_2}\right).
	\end{align*}

	\subsection{Microlocal analysis of the wave kernel on conic manifolds}
	
The linear wave propagation on smooth manifolds are described by \cite{Ho_FIO_1, Ho_FIO_2}'s classical theorem of propagation of singularities. In this manner, \cite{CT82a, CT82b} elucidated the diffraction on conic manifolds of product type.
	\begin{theorem} \label{conic}
		Let $C(\Theta)$ be a conic manifold with metric \eqref{metric} and $u(0,r, \theta) \in \mathcal{D}^k$ for some $k$. Suppose $$u(t) := \exp\left(i\sqrt{\Delta_{C(\Theta)}}t\right)u(0)$$ solves the half wave equation $$\partial_t u = i\sqrt{\Delta_{C(\Theta)}}u$$ on $C(\Theta)$ with initial data $u(0)$. Then the following holds.
		\begin{itemize}
			\item The half wave $u(t)$ is smooth in Region III. Namely, 
			if $\operatorname{supp} u(0)$ lies in the truncated cone $C_{0, r_2}(\Theta)$, then $u(t) \in \mathcal{D}_{\mathrm{loc}}^{\infty}(C_{0, t - r_2}(\Theta))$ for all $t > r_2$.
			\item The propagation of waves in the interior of $C(\Theta)$ obeys the laws of geometrical optics.
			That is, if $v \in T^*_{(r_2, \theta_2)}(C(\Theta))$ is a unit covector such that the Hamiltonian flow \( H_t(v) \) never reaches the point \( P \) (equivalently \( v \neq -dr \)), then  
			\begin{equation}\label{conic propagation inner}
				(r_2, \theta_2, v) \in \mathrm{WF}(u(0))\Longleftrightarrow H_t((r_2, \theta_2, v) ) \in \mathrm{WF}(u(t)) \quad \forall t > 0.
			\end{equation}
			\item The half wave $u(t)$ is smooth near $P$ provided no rays pass through a neighbourhood of $P$. That is, 
			suppose $u(0)|_{C_{0,a}(\Theta)} \in \mathcal{D}^{\infty}_{\mathrm{loc}}(C_{0,a}(\Theta))$ such that for all $t \geq 0$, $\pi_{C(\Theta)}(H_t(\mathrm{WF}(u(0))))\cap C_{0,a}(\Theta)=\emptyset$, then  
			\begin{equation}\label{conic propagation conical}
				u(t) \in \mathcal{D}^{\infty}_{\mathrm{loc}}(C_{0,a}(\Theta))\quad \forall t \geq 0.
			\end{equation}
			\item No singularities arise from the vertex $P$ independently. Namely, there holds for $t > r_2 > 0$ that
			\begin{displaymath}
				\exists \theta_2\in\Theta,\,  (r_2,\theta_2,-dr) \in \mathrm{WF}(u(0))\Longleftrightarrow\exists \theta_1\in\Theta,\,(t-r_2,\theta_1,-dr) \in \mathrm{WF}(u(t)).
			\end{displaymath}
		\end{itemize}
	\end{theorem}
	
	\begin{remark}
		We remark that the propagation of singularities on conic manifolds of non-product type was established by Melrose and Wunsch \cite{MW04}.
	\end{remark}

	Furthermore, we review the microlocal structure of the wave kernel on conic manifolds. We first recall the notion of conormal distributions.
	
	\begin{definition}
		Let $X$ be a smooth $n$-dimensional manifold and $Y \subset X$ a smooth embedded submanifold of codimension $k\ (1 \le k \le n-1)$. Denote by $N^*Y \subset T^*X \setminus 0$ the conormal bundle of $Y$
		\[
		N^*Y := \{ (y,\xi) \in T^*X : \xi(v)=0 \text{ for all } v \in T_yY\}.
		\]
		For $m \in \mathbb R$, the space of conormal distributions of order $m$ associated with $N^*Y$, denoted by $I^{m}(X,N^*Y)$, consists of all distributions $u$ such that
		\begin{itemize}
			\item For any $p\in Y$, there exists neighbourhood $U\ni p$ with local coordinates $x=(x',x'') \in \mathbb R^{k} \times \mathbb R^{n-k}$ in which $Y \cap U=\{x'=0\},$
			\item $u$ has an oscillatory integral representation in $U$,
			\[
			u(x)= (2\pi)^{-k} \int_{\mathbb R^{k}} e^{ i\, x' \cdot \xi }\, b(x'',\xi)\, d\xi,
			\]
			where the amplitude $b\in S^{m + \frac{n}{4} - \frac{k}{2}}(\mathbb R^{n-k}_{x''} \times \mathbb R^{k}_{\xi})$.
		\end{itemize}
		The principal symbol of $u\in I^{m}(X,N^*Y)$ is locally given by
		\[
		\sigma_p(u)(x'',\xi) = (2\pi)^{-k/2} \, b_{m}(x'',\xi)\, |dx' dx''|^{1/2}\, |d\xi|^{1/2}
		\quad \text{restricted to } N^*Y,
		\]
		and lies in the class
		\begin{equation}\label{pricipal symbol of conomal}
			\sigma_p(u) \in S^{m+n/4}(N^*Y)/S^{m+n/4-1}(N^*Y).
		\end{equation}
	\end{definition}
	
	Ford and Wunsch \cite{FW17} used the framework of conormal distributions to analyze the structure of the wave kernel on conic manifolds.
	
	\begin{proposition}[\mbox{\cite[Proposition 2.1]{FW17}}]\label{conormal diffracted wave}
		Let $C(\Theta)$ be a conic manifold with metric \eqref{metric}. Assume $(r_1,\theta_1)\in \mathcal{D}_{t}^{(r_2,\theta_2)}\setminus \mathcal{F}_{t}^{(r_2,\theta_2)}$. Then, we have the following:
		\begin{itemize}
			\item Near $(t,(r_1,\theta_1),(r_2,\theta_2))$ in $\mathbb{R}_+ \times C(\Theta) \times C(\Theta)$, the kernel $K_{\sin t\lambda/\lambda}$ is a conormal distribution associated with the conormal bundle of the diffracted front
			\[\mathcal{D}:= \{ (t,(r_1,\theta_1),(r_2,\theta_2)) : t = r_1 + r_2 \},\]
			of order $-\frac{5}{4}-\frac{n}{2}$:
			\begin{equation}\label{conormal order}
				K_{\sin t\lambda/\lambda} \;\in\; I^{-\frac{5}{4}-\frac{n}{2}}\!\Bigl(\mathbb{R}_+ \times C(\Theta) \times C(\Theta),\; N^*\mathcal{D}\Bigr);
			\end{equation}
			\item The leading order singularity of $K_{\sin t\lambda/\lambda}$ is
			\begin{equation}\label{leading order singularity}
				\frac{1}{\pi}(r_1r_2)^{-n/2}\left[-\frac{\pi}{2}\sin(\pi\nu)H(\epsilon_d)-\cos(\pi\nu)\log|\epsilon_d|\right],
			\end{equation}
			where
			\[
			\epsilon_d:=\mathrm{sgn}(t^2-(r_1+r_2)^2)\cdot\left|\frac{t^2-(r_1+r_2)^2}{r_1r_2}\right|^{1/2};
			\]
			\item $K_{\sin t\lambda/\lambda}$ admits a local oscillatory integral representation
			\[K_{\sin t\lambda/\lambda}(t,(r_1,\theta_1),(r_2,\theta_2))=\int_{\mathbb{R}_{\xi}}e^{\,i(r_1 + r_2 - t)\xi}\,a(t,r_1,\theta_1,r_2,\theta_2;\xi)\, d\xi,\]
			where the amplitude $a$ is a classical symbol of order $-1$ in $\xi$;
			\item $a$ has an asymptotic
			\begin{multline}\label{principal symbol}
				a(t,r_1,\theta_1,r_2,\theta_2;\xi)\;\equiv\;\frac{(r_1 r_2)^{-\frac{n}{2}}}{2\pi}\,\frac{\chi_a(\xi)}{2\,|\xi|} \Big(H(-\xi)\,K[e^{-i\pi \nu}](\theta_1,\theta_2)\\+ H(\xi)\,K[e^{i\pi \nu}](\theta_1,\theta_2)\Big)
				\pmod{S^{-3/2}},
			\end{multline}
			where:
			\begin{itemize}
				\item $\chi_a\in C^\infty(\mathbb{R})$ is a cutoff with $\chi_a(\xi)=0$ for $|\xi|<1$ and $\chi_a(\xi)=1$ for $|\xi|>2$;
				\item $K[e^{\mp i\pi \nu}](\theta_1,\theta_2)$ denotes the Schwartz kernel (in $\Theta\times\Theta$) of the operator $e^{\mp i\pi \nu}$;
				\item the remainder $S^{-3/2}$ is a symbol class of order $-3/2$ in $\xi$, smooth in $(t,(r_1,\theta_1),(r_2,\theta_2))$.
			\end{itemize}
		\end{itemize}
	\end{proposition}
	
	When $\Theta = \mathbb{R}/2\pi\beta\mathbb{Z}$, \cite[(4.11)]{CT82b} calculates that
	\begin{equation}\label{wave}
		\sin\left( s\nu\right)(\delta_{\theta_2})(\theta_1)=\frac{\beta^{-1}}{2\pi}\bigg(\frac{\sin \beta^{-1}(s+\theta)}{2-2\cos\beta^{-1}(s+\theta)}+\frac{\sin \beta^{-1}(s-\theta)}{2-2\cos\beta^{-1}(s-\theta)}\bigg),
	\end{equation}
	where $\theta=\theta_1-\theta_2$. It follows from the method of characteristics that
	\begin{equation}\label{wave_cos}
		\cos \left(s\nu\right)(\delta_{\theta_2})(\theta_1)=\frac{1}{2}[\delta_{s}(\theta)+\delta_{-s}(\theta)],
	\end{equation}
	which vanishes on $\{\theta_1:\theta\not\equiv \pm s\pmod{2\pi\beta}\}$. Hence, by \eqref{leading order singularity} and \eqref{wave_cos}, the leading order singularity of  $$\frac{\sin\left(t\sqrt{\Delta_{C(\Theta)}}\right)}{\sqrt{\Delta_{C(\Theta)}}}(\delta_{(r_2,\theta_2)})(r_1,\theta_1)$$ 
	on $\mathcal{D}^{O}_{t}\setminus\mathcal{F}_{t}^{O}$ is 
	\begin{align}\label{pricipal term on I}
		u^{\mathrm{prin}}_{(r_2,\theta_2)}(r_1,\theta_1):=-\frac{1}{\pi}(r_1r_2)^{-1/2}&\bigg[\frac{\pi}{2}\sin(\pi\nu)(\delta_{(\theta_2)})(\theta_1)H(-\epsilon_d)\nonumber\\
		&\qquad\qquad+\cos(\pi\nu)(\delta_{(\theta_2)})(\theta_1)\log|\epsilon_d|\bigg]\nonumber\\
		=-\frac{1}{2}(r_1r_2)^{-1/2}&\sin(\pi\sqrt{\Delta_{\Theta}})(\delta_{(\theta_2)})(\theta_1)H(t-r_1-r_2).
	\end{align}
	In terms of Sobolev regularities of conormal distributions (for example \cite[Theorem 18.2.8]{Ho_PDO_3}), by \eqref{principal symbol}, the remainder satisfies
	\begin{equation}\label{regularity of remained term of wave kernel}
		\frac{\sin\left(t\sqrt{\Delta_{C(\Theta)}}\right)}{\sqrt{\Delta_{C(\Theta)}}}(\delta_{(r_2,\theta_2)})(r_1,\theta_1)-u^{\mathrm{prin}}_{(r_2,\theta_2)}(r_1,\theta_1)\in \bigcap_{\epsilon>0}H^{1-\epsilon}(U),
	\end{equation} 
	where $U\subset C(\Theta)$ and $U\cap \mathcal{F}_{t}^{O}=\emptyset$.

	\section{The diffracted waves by a corner in $\mathbb{R}^2$}\label{sec3}

	As explained in Section \ref{sec1}, we have reformulated the obstacle problem \eqref{eq_1} as the parameter identification problem on a cone over a smooth manifold without boundary, as given by the Cauchy problem \eqref{conical}.  Recall $O=(r, \phi)$. Let $x=(r_1,\theta_1)$ and define $\alpha:=\theta_1-\phi$. By the principle of superposition, the solution of  \eqref{conical} is of the form
	\begin{multline*}
		u_{\Omega}(t,r_1,\theta_1,O)=\\
		\frac{\sin\left(t\sqrt{\Delta_{C(I_{4\pi\beta})}}\right)}{\sqrt{\Delta_{C(I_{4\pi\beta})}}}(\delta_{(r,\phi)})(r_1,\theta_1)+\frac{\sin\left(t\sqrt{\Delta_{C(I_{4\pi\beta})}}\right)}{\sqrt{\Delta_{C(I_{4\pi\beta})}}}(\delta_{(r,-\phi)})(r_1,\theta_1).
	\end{multline*}
	In  $\mathcal{D}_{t}^{O}\setminus\mathcal{F}_{t}^{O}$, $u_{\Omega}$ is the diffracted waves. The principal terms of $u_{\Omega}$ satisfy
	\begin{align}\label{principal diffracted}
		u_{\Omega}(t,r_1,\theta_1,O)-\left[u^{\mathrm{prin}}_{(r,\phi)}(r_1,\theta_1)+u^{\mathrm{prin}}_{(r,-\phi)}(r_1,\theta_1)\right]\in \bigcap_{\epsilon>0}H^{1-\epsilon}(U),
	\end{align}
	where $U\subset \Omega$ and $U\cap \mathcal{F}_{t}^{O}=\emptyset$. Hence, by \eqref{pricipal term on I}, the principal diffracted waves in $\Omega$ take the form
	\begin{align*}
		&u_{\mathcal{D},P}^{\mathrm{prin}}(t, x, O):=u^{\mathrm{prin}}_{(r,\phi)}(r_1,\theta_1)+u^{\mathrm{prin}}_{(r,-\phi)}(r_1,\theta_1)\nonumber\\
		&=-\frac{1}{2}(rr_1)^{-1/2}H(t-r-r_1)\left[\sin\left(\pi\sqrt{\Delta_{I_{4\pi \beta}}}\right)\delta_{\phi}+\sin\left(\pi\sqrt{\Delta_{I_{4\pi \beta}}}\right)\delta_{-\phi}\right],
	\end{align*}
	which is \eqref{asy_0}. Substituting \eqref{wave} into \eqref{asy_0}, we obtain
	\begin{multline}\label{diffracted wave}
		u_{\mathcal{D},P}^{\mathrm{prin}}(t, x, O)=-\frac{(rr_1)^{-1/2}}{16\pi\beta}H(t-r-r_1)\bigg[\cot\left(\frac{\pi+\alpha}{4\beta}\right)+\cot\left(\frac{\pi-\alpha}{4\beta}\right) \\+\cot\left(\frac{\pi+\alpha+2\phi}{4\beta}\right)+\cot\left(\frac{\pi-\alpha-2\phi}{4\beta}\right)\bigg].
	\end{multline}
	
	We introduce the shorthand notations 
	\begin{align}\label{4.1}
		S_{\phi,\lambda}(\alpha)&:=\lambda\sum_{j=1}^{4}S^{(j)}_{\phi,\lambda}(\alpha), \end{align} with $\lambda:= (2\beta)^{-1}$ and \begin{align*} S^{(j)}_{\phi,\lambda}(\alpha)&:=\cot\bigg(\frac{\pi+(-1)^{j-1}\alpha+(-1)^{j-1}2\left[(j-1)/2\right]\phi}{2}\lambda\bigg).\notag
		\end{align*}
In terms of \eqref{4.1}, the oscillatory part of \eqref{diffracted wave} corresponds to $S_{\phi, (2\beta)^{-1}}(\alpha)$.

Note that $S^{(j)}_{\phi,\lambda}(\alpha)$ is a periodic function with minimal positive period $2\pi/\lambda$. Hence, $2\pi/\lambda$ is also a period of $S_{\phi,\lambda}(\alpha)$.
	
	\begin{remark}\label{concave corner which diffraction vanishes}
		If $\lambda\in\mathbb{Z}_+$, 
		\begin{equation*}
			\cot\left(\frac{\pi+\alpha}{2}\lambda\right)+\cot\left(\frac{\pi-\alpha}{2}\lambda\right)=\cot\left(\frac{\pi+\alpha}{2}\lambda\right)+\cot\left(-\frac{\pi+\alpha}{2}\lambda\right)=0.
		\end{equation*}
		
		If $\lambda\in\frac{1}{2}\mathbb{Z}_+$ and $\phi=\pi/(2\lambda)$, 
		\begin{align*}
			S_{\phi,\lambda}(\alpha)&=2\lambda\left[\cot\left(\left(\pi+\alpha\right)\lambda\right)+\cot\left(\left(\pi-\alpha\right)\lambda\right)\right]\\
			&=2\lambda\left[\cot\left(\left(\pi+\alpha\right)\lambda\right)-\cot\left(\left(\pi+\alpha\right)\lambda\right)\right]=0.
		\end{align*}
		Consequently, \eqref{4.1} vanishes in both cases, indicating that no diffraction occurs.
	\end{remark}
	
We are now in the position to prove that the values of	$S_{\phi,\lambda}(\alpha)$ uniquely determine the parameters $\phi$ and $\lambda$. More precisely,
	
	\begin{theorem}\label{T4.1}
		Let $\lambda,\tilde{\lambda}\in (1/2, \infty) \setminus \frac{1}{2}\mathbb{Z}$ and $0<\phi,\tilde{\phi}<\pi/\lambda$. If there exist $\epsilon>0$ and $\alpha_0 \in \mathbb{R}$ such that  
		\begin{equation}\label{eqn : data of S}
			S_{\phi,\lambda}(\alpha)=S_{\tilde{\phi},\tilde{\lambda}}(\alpha), \quad
		\forall	\alpha\in(\alpha_0-\epsilon,\alpha_0+\epsilon),
	\end{equation} then we have
		\begin{equation*}
	  \lambda=\tilde{\lambda},\quad \phi=\tilde{\phi}.
		\end{equation*}
	\end{theorem}
	\begin{proof}

Consider the meromorphic extension of $S_{\phi, \lambda}$ to $\mathbb{C}$,
\begin{align}\label{3.1.1}
	\mathcal{S}_{\phi,\lambda}(z)&:=\lambda\sum_{j=1}^{4}\mathcal{S}^{(j)}_{\phi,\lambda}(z),\quad z \in \mathbb{C},	\end{align} where each term reads \begin{align*} \mathcal{S}^{(j)}_{\phi,\lambda}(z)&:=\cot\bigg(\frac{\pi+(-1)^{j-1}z+(-1)^{j-1}2[(j-1)/2]\phi}{2}\lambda\bigg),\quad z \in \mathbb{C}. \notag
\end{align*} This is a meromorphic function having only countably many isolated poles. Denote by $\mathcal{P}(f)$   the collection of poles of a meromorphic function $f(z)$. Then $\mathbb{C}\setminus(\mathcal{P}(\mathcal{S}_{\phi,\lambda})\cup\mathcal{P}(\mathcal{S}_{\tilde{\phi},\tilde{\lambda}}))$ is connected. 

Then \eqref{eqn : data of S} implies that $$\mathcal{S}_{\phi,\lambda}(z)=\mathcal{S}_{\tilde{\phi},\tilde{\lambda}}(z),\quad \forall z\in(\alpha_0-\epsilon,\alpha_0+\epsilon).$$  For two holomorphic functions defined on a connected domain, if they agree on a set with an accumulation point in the domain, then they must coincide identically. Hence, we have $\mathcal{S}_{\phi,\lambda}(z)=\mathcal{S}_{\tilde{\phi},\tilde{\lambda}}(z)$ for all $z\in\mathbb{C}$. Therefore,  $S_{\phi,\lambda}(\alpha)=S_{\tilde{\phi},\tilde{\lambda}}(\alpha)$ for all $\alpha \in \mathbb{R}$.
It follows that the minimal positive periods $T_{\min}$ and $\tilde{T}_{\min}$ of $S_{\phi,\lambda}(\alpha)$ and $S_{\tilde{\phi},\tilde{\lambda}}(\alpha)$ are equal.
With regard to the minimal positive period, we assert  
\begin{lemma}\label{T4.1_lem}
	The minimal positive period of $S_{\phi,\lambda}(\alpha)$ is $T_{\min} =2\pi/\lambda$ when $\phi\neq\pi/(2\lambda)$.
\end{lemma}
We assume that Lemma \ref{T4.1_lem} is true, and provide the proof later in this section. Now, we divide the proof of Theorem \ref{T4.1} in the following cases.

\subsection*{Case 1}
If $\phi\neq\pi/(2\lambda)$ and $\tilde{\phi}\neq\pi/(2\tilde{\lambda})$, we obtain $\lambda=\tilde{\lambda}$. 

Moreover, $S_{\phi,\lambda}(\alpha)=S_{\tilde{\phi},\tilde{\lambda}}(\alpha)$ also yields that
\begin{multline*}
	\cot\!\left(\frac{\pi+\alpha+2\phi}{2}\lambda\right)+\cot\!\left(\frac{\pi-\alpha-2\phi}{2}\lambda\right) \\
	=\cot\!\left(\frac{\pi+\alpha+2\tilde{\phi}}{2}\lambda\right)+\cot\!\left(\frac{\pi-\alpha-2\tilde{\phi}}{2}\lambda\right), \quad \forall \alpha \in \mathbb{R},
\end{multline*}
reducing to
\begin{equation*}
	\frac{\sin(\pi\lambda)}{\cos(\pi\lambda)-\cos((\alpha+2\phi)\lambda)}=\frac{\sin(\pi\lambda)}{\cos(\pi\lambda)-\cos((\alpha+2\tilde{\phi})\lambda)}, \quad \forall \alpha \in \mathbb{R}.
\end{equation*}
Since $\lambda\notin\mathbb{Z},$ this is equivalent to $2\lambda\phi\equiv 2\lambda\tilde{\phi}\pmod{2\pi},$ i.e. $\phi\equiv\tilde{\phi}\pmod{\pi/\lambda}.$

\subsection*{Case 2}
If $\phi=\pi/(2\lambda)$ and $\tilde{\phi}=\pi/(2\tilde{\lambda}),$ \eqref{4.1} becomes
\begin{align*}
	S_{\phi,\lambda}(\alpha)=2\lambda\left[\cot\left(\left(\pi+\alpha\right)\lambda\right)+\cot\left(\left(\pi-\alpha\right)\lambda\right)\right]
	=\frac{4\lambda\sin(2\pi\lambda)}{\cos(2\alpha\lambda)-\cos(2\pi\lambda)}.
\end{align*}
Since $\lambda\notin\tfrac{1}{2}\mathbb{Z}$ as assumed, we have $S_{\phi,\lambda}(\alpha)\not\equiv 0$ with minimal period $T_{\min}=\pi/\lambda$ and then necessarily $\lambda=\tilde{\lambda}.$ 

\subsection*{Case 3}
If $\phi=\pi/(2\lambda)$ and $\tilde{\phi}\neq \pi/(2\tilde{\lambda})$, Lemma \ref{T4.1_lem} gives $\tilde{\lambda}=2\lambda$. Therefore, it follows from $S_{\phi,\lambda}(\alpha)=S_{\tilde{\phi},\tilde{\lambda}}(\alpha)$ that
\[
\cot\!\left(\frac{\pi+\alpha+2\tilde{\phi}}{2}\,\tilde{\lambda}\right)
=\cot\!\left(\frac{-\pi+\alpha+2\tilde{\phi}}{2}\,\tilde{\lambda}\right), \quad \forall \alpha\in\mathbb{R}.
\]
This is equivalent to $\tilde{\lambda}\pi\equiv 0 \pmod{\pi},$ i.e. $\tilde{\lambda}\in\mathbb{Z},$ which is a contradiction. 

Therefore, mapping \eqref{eqn : observation at l} is injective, if Lemma \ref{T4.1_lem} is true.
\end{proof}

It remains to prove Lemma \ref{T4.1_lem}.

\begin{proof}[Proof of Lemma \ref{T4.1_lem}]

Recall  the meromorphic extension $\mathcal{S}_{\phi,\lambda}$ in \eqref{3.1.1} of $S_{\phi,\lambda}$. We denote by $\mathcal{P}(\mathcal{S})$ the set of the poles of a meromorphic function $\mathcal{S}$ on $\mathbb{C}$ and also use the shorthand $\mathcal{P}_j := \mathcal{P} (\mathcal{S}^{(j)}_{\phi,\lambda})$ for simplicity. In particular, we see that  
 \begin{align}\label{trivial poles of P_j}  \left\{
 	\begin{aligned}
   \mathcal{P}_1 &= \left\{  - \pi + 2m\pi/\lambda : m \in \mathbb{Z}  \right\} \\
 \mathcal{P}_2 &= \left\{    \pi + 2m\pi/\lambda : m \in \mathbb{Z}  \right\}\\
  \mathcal{P}_3 &= \left\{   - \pi - 2\phi + 2m\pi/\lambda : m \in \mathbb{Z}  \right\}\\
   \mathcal{P}_4 &= \left\{     \pi - 2\phi + 2m\pi/\lambda : m \in \mathbb{Z}  \right\}
   \end{aligned}\right.
\end{align}
and
$\mathcal{P}(\mathcal{S}_{\phi,\lambda})$ is invariant under the translation by the minimal positive period $T_{\min}$ along the real axis,
\begin{equation}\label{eqn : translation invariance of poles}
	\mathcal{P}(\mathcal{S}_{\phi,\lambda})=T_{\min}+\mathcal{P}(\mathcal{S}_{\phi,\lambda}).
\end{equation} 

The proof is then structured by classifying cases based on the number and distribution of the poles of $\mathcal{S}_{\phi,\lambda}(z)$ in the interval of length $T=2\pi/\lambda$.  Aside from the translation invariance \eqref{eqn : translation invariance of poles}, the proof also utilizes the pigeonhole principle. Specifically, if $m$ objects are distributed among $k$ containers, then at least one container must contain no fewer than $\left\lceil \frac{m}{k} \right\rceil$ objects.

\subsection*{Case 1}	 We assume that $\{\mathcal{P}_k\}_{k=1, 2, 3, 4}$ are pairwisely disjoint. This also implies that $\mathcal{P}(\mathcal{S}_{\phi,\lambda})$  has precisely $4$ distinct elements  in $[0,2\pi/\lambda)$ by \eqref{eqn : translation invariance of poles}.

It is clear that $T=2\pi/\lambda$ is a period of $S_{\phi,\lambda}(\alpha)$. Since $\mathcal{P}(\mathcal{S}_{\phi,\lambda})\neq\emptyset$, $S_{\phi,\lambda}(\alpha)$ can not be a constant. Thus, $T=2\pi/\lambda$ is divisible by $T_{\min}$ and one can find some $N\in\mathbb{Z}_+$ such that $T_{\min}=2\pi/(N\lambda)$. It suffices to prove  $N = 1$. We assume $N>1$ and proceed by contradiction.

 \subsubsection*{Subcase 1.1}
If $N$ is odd, we divide $[0,2\pi/\lambda)$ into $N$ consecutive subintervals
\[
\left\{\left[(k - 1) \frac{2\pi}{N \lambda}, k  \frac{2\pi}{N \lambda}\right)   :    {k=1, \cdots, N}\right\}.
\]
By the pigeonhole principle, at least one of these subintervals contains $\left\lceil 4/N \right\rceil = \left\lfloor 4/N \right\rfloor + 1$ poles. By periodicity, this would force the entire interval $[0,2\pi/\lambda)$ to contain at least
$N + N \left\lfloor 4/N \right\rfloor > 4$
poles, which is impossible.

 \subsubsection*{Subcase 1.2}
For even $N$, we have that
\[
T_0:=\frac{\pi}{\lambda}=\frac{N}{2}\times\frac{2\pi}{N\lambda}
\]
is also a period of $S_{\phi,\lambda}(\alpha)$. However, we can show that this is not possible.

Otherwise, we first assert that there must exist $\{i, j\} \subset \{1, 2, 3, 4\}$ such that 
\begin{align}\label{eqn : P_i = P_j + T}\mathcal{P}_i+ T_0 = \mathcal{P}_j .\end{align}
 
\begin{proof}[Proof of \eqref{eqn : P_i = P_j + T}] For any $i = 1, 2, 3, 4$, one may take some $\alpha_i \in \mathcal{P}_i$, and rewrite
  $$\mathcal{P}_i = \{\alpha_i + 2 m \pi / \lambda : m \in \mathbb{Z}\}.$$ 
 If there is some $\beta \in \mathcal{P}_i \cap (\mathcal{P}_i + T_0)$, then it follows that there must exist some $m_1, m_2 \in \mathbb{Z}$ such that $$\beta = \alpha_i + 2 m_1 \pi / \lambda = \alpha_i + T_0 + 2 m_2 \pi / \lambda.$$ This implies
  $T_0 = 2 (m_1 - m_2) \pi / \lambda,$ which contradicts $T_0=\pi/\lambda$. Therefore, we have
	$$\mathcal{P}_i \cap (\mathcal{P}_i + T_0) = \emptyset.$$
Together with \eqref{eqn : translation invariance of poles}, this yields some $j \neq i$ and $\alpha_j \in \mathbb{R}$ such that 
$$\alpha_j \in (\mathcal{P}_i + T_0) \cap \mathcal{P}_j \neq\emptyset.$$ We thus can rewrite 
	\begin{align*}
	  \mathcal{P}_j &= \{\alpha_j + 2 m \pi / \lambda : m \in \mathbb{Z}\},
	\\
	 \mathcal{P}_i &= \{\alpha_j - T_0 + 2 m \pi / \lambda : m \in \mathbb{Z}\}.
	\end{align*}
	This shows \eqref{eqn : P_i = P_j + T}.
\end{proof}

\begin{remark}
Since $T_0 =  \pi/\lambda$,  \eqref{eqn : P_i = P_j + T} is actually equivalent to $$\mathcal{P}_i = \mathcal{P}_j+ T_0.$$ 
\end{remark}

 Next, we examine all scenarios within Subcase 1.2, and apply  \eqref{eqn : P_i = P_j + T} to derive a contradiction for every possible configuration.

\subsubsection*{Subsubcase 1.2.1} When $\mathcal{P}_1 = \mathcal{P}_2 + T_{0}$ or $\mathcal{P}_3 = \mathcal{P}_4 + T_{0}$, we use \eqref{trivial poles of P_j}
	to obtain  $$-\pi \equiv \pi + T_{0} \pmod{ 2\pi / \lambda },$$ and thus
	\begin{equation*}\label{case.1.2.1}
		\lambda\equiv\frac{1}{2}\pmod{1}; \quad \mbox{equivalent,  $\lambda = {(2M+1)}/{2}$ for some $M \in \mathbb{Z}_+$. }
	\end{equation*}
  Plugging it into \eqref{4.1} gives
		\begin{eqnarray}\label{4.2}
			S_{\phi,\lambda}(\alpha)
			 =  2\lambda\left[\frac{1}{\sin((\pi+\alpha)\lambda)}+\frac{1}{\sin((\pi+\alpha+2\phi)\lambda)}\right].
	\end{eqnarray} 
    This implies that $$S_{\phi,\lambda}(\alpha) = - S_{\phi,\lambda}(\alpha+\pi/\lambda).$$ Since $S_{\phi,\lambda}$ is $(\pi/\lambda)$-periodic, we also have that $$S_{\phi,\lambda}(\alpha) = S_{\phi,\lambda}(\alpha+\pi/\lambda).$$ Hence, $$S_{\phi,\lambda}(\alpha)\equiv 0, \quad \mbox{ $\forall \alpha\in \mathbb{R}$.}$$ 
    Combining this with \eqref{4.2} yields
    \[-\sin\left((\pi+\alpha)\lambda\right)=\sin\left((\pi+\alpha+2\phi)\lambda\right),\quad \mbox{$\forall \alpha \in \mathbb{R}$.} \]
    This amounts to 
    \[2\phi\lambda \equiv \pi   \pmod{2\pi}  \Longleftrightarrow  \phi \equiv \pi/(2\lambda) \pmod{\pi/\lambda},\]
    which contradicts the assumption $\phi\neq\pi/(2\lambda)$.

\subsubsection*{Subsubcase 1.2.2}    
When $\mathcal{P}_1 = \mathcal{P}_3 + T_{0}$ or $\mathcal{P}_2 = \mathcal{P}_4 + T_{0}$,  we likewise use \eqref{trivial poles of P_j} to obtain 
	\begin{equation}\label{case.1.2.2}
		\phi\equiv\frac{\pi}{2\lambda}\pmod{\pi/\lambda},
	\end{equation}
	which contradicts the assumption $\phi\neq \pi/(2\lambda)$.

\subsubsection*{Subsubcase 1.2.3}   When $\mathcal{P}_1 = \mathcal{P}_4 + T_{0}$ or $\mathcal{P}_2 = \mathcal{P}_3 + T_{0}$, then \eqref{trivial poles of P_j} gives that
\begin{equation}\label{case.1.2.3}
	\phi\equiv\pm\pi+ \frac{\pi}{2\lambda}\pmod{\pi/\lambda},
\end{equation}
where $+$ and $-$ correspond to $\mathcal{P}_1 = \mathcal{P}_4 + T_{0}$ and $\mathcal{P}_2 = \mathcal{P}_3 + T_{0}$ respectively. Without loss of generality, we assume $\mathcal{P}_2 = \mathcal{P}_3 + T_{0}$. Using \eqref{case.1.2.3}, we reduce \eqref{4.1} to
\begin{equation*}
	S_{\phi,\lambda}(\alpha)=\lambda\bigg[\cot\left(\frac{\pi+\alpha}{2}\lambda\right)-\tan\left(\frac{3\pi-\alpha}{2}\lambda\right)+\frac{2}{\sin\left((\pi-\alpha)\lambda\right)}\bigg].
\end{equation*}
By  $S_{\phi,\lambda}(\alpha)=S_{\phi,\lambda}(\alpha+\pi/\lambda)$, we obtain
\begin{equation}\label{eqn: case.1.2.3}
	\frac{1}{\sin\left(\left(\alpha+\pi\right)\lambda\right)}-\frac{1}{\sin\left(\left(\alpha-\pi\right)\lambda\right)}=\frac{1}{\sin\left(\left(\alpha-\pi\right)\lambda\right)}-\frac{1}{\sin\left(\left(\alpha-3\pi\right)\lambda\right)}.
\end{equation}

Consider the \(2\pi/\lambda\)-periodic function
\begin{equation}\label{period function in case 1.2.3}
	F(\alpha):=\frac{1}{\sin\big(\alpha\lambda\big)}-\frac{1}{\sin\big((\alpha-2\pi)\lambda\big)}.
\end{equation}
By \eqref{eqn: case.1.2.3}, the function \(F(\alpha)\) is also \(2\pi\)-periodic. Observe that \(F(\alpha)\equiv C_0\) if and only if
\[
2\pi\lambda \equiv 2\pi \pmod{2\pi},
\]
and hence \(F(\alpha)\not\equiv C_0\) under the assumption \(\lambda\notin\mathbb{Z}\). It follows that $2\pi$ is divisible by $T_{\min}^F$, the minimal positive period of \(F(\alpha)\). Since \(2\pi/\lambda\) is a period of \(F(\alpha)\), $2\pi/\lambda T_{\min}^F\in\mathbb{Z}$ and hence $$\lambda=\frac{2\pi}{T_{\min}^F}\times \frac{\lambda T_{\min}^F}{2\pi}\in\mathbb{Q}.$$ 
We may write
\[
\lambda=\frac{p}{q}, \qquad \text{with } p,q\in\mathbb{Z}_+ \text{ and } (p,q)=1.
\]
Since \(2\pi\) and \(2\pi/\lambda\) are periods of \(F(\alpha)\) and \((p,q)=1\), it follows that \(2\pi/p\) is also a period of \(F(\alpha)\).

If \(q\) is even, then 
\[
\frac{\pi}{\lambda}=\frac{q}{2}\times\frac{2\pi}{p}
\]
is also a period of \(F(\alpha)\). Noting $F(\alpha) = -F\left(\alpha+\pi/\lambda\right)$, we conclude that \(F(\alpha)\equiv 0\). Substituting this into \eqref{period function in case 1.2.3} gives
\[
\sin\big(\alpha\lambda\big)=\sin\big((\alpha-2\pi)\lambda\big),
\]
which implies
\[
2\pi\lambda \equiv 2\pi \pmod{2\pi}.
\]
However, this is impossible under the standing assumption \(\lambda\notin\mathbb{Z}\).

If $q>1$ is odd, we partition $[0,2\pi/\lambda)$, as in Subcase 1.1, into 
\[
\left\{\left[(m - 1) \frac{2\pi}{q \lambda}, m  \frac{2\pi}{q \lambda}\right)   :    {m=1, \cdots, q}\right\}.
\]
Since \(\lambda\notin\frac{1}{2}\mathbb{Z}\), \(F(\alpha)\) has $4$ poles within one \(2\pi/\lambda\)-period. By the pigeonhole principle, there must be a subinterval containing $\left\lceil 4/N \right\rceil = \left\lfloor 4/N \right\rfloor + 1$ poles. The number of poles in $[0,2\pi/\lambda)$ is at least
$N + N \left\lfloor 4/N \right\rfloor>4$, which is impossible.

\subsection*{Case 2}    
There exist distinct $\mathcal{P}_i$ and $\mathcal{P}_j$ having a non-empty intersection.

First of all, we claim that
\begin{equation}\label{case.2}
	\phi \equiv \pm \pi   \pmod{\pi/\lambda}.
\end{equation}
\begin{proof}[Proof of \eqref{case.2}]
	 In this case, there exist $\{i, j\} \subset \{1, 2, 3, 4\}$ such that $\mathcal{P}_i = \mathcal{P}_j$. We next examine each possible scenario:
\begin{itemize}
	\item $\mathcal{P}_1=\mathcal{P}_2$ implies that $-\pi\equiv\pi \pmod{2\pi/\lambda}$, which contradicts $\lambda\notin\mathbb{Z}$ assumed in Theorem \ref{T4.1};
	\item $\mathcal{P}_1=\mathcal{P}_3$ implies that $-\pi\equiv-\pi-2\phi \pmod{2\pi/\lambda}$, i.e., $\lambda\phi/\pi\in\mathbb{Z}$, which contradicts $ 0<\phi<\pi/\lambda$ assumed in Theorem \ref{T4.1};
	\item $\mathcal{P}_1=\mathcal{P}_4$ implies that $-\pi\equiv\pi-2\phi \pmod{2\pi/\lambda}$, i.e., $\phi\equiv\pi \pmod{\pi/\lambda}$;
	\item $\mathcal{P}_2=\mathcal{P}_3$ implies that $\pi\equiv-\pi-2\phi \pmod{2\pi/\lambda}$, i.e., $\phi\equiv-\pi \pmod{\pi/\lambda}$;
	\item $\mathcal{P}_2=\mathcal{P}_4$ implies that $\pi\equiv\pi-2\phi \pmod{2\pi/\lambda}$, i.e., $\lambda\phi/\pi\in\mathbb{Z}$, which contradicts $ 0<\phi<\pi/\lambda$ assumed in Theorem \ref{T4.1};
	\item $\mathcal{P}_3=\mathcal{P}_4$ implies that $-\pi-2\phi\equiv\pi-2\phi \pmod{2\pi/\lambda}$, which contradicts $\lambda\notin\mathbb{Z}$ assumed in Theorem \ref{T4.1}.
\end{itemize}

\end{proof}

Without loss of generality, we restrict attention to the case
\begin{equation}\label{case.2. real}
	\phi \equiv \pi \pmod{\pi/\lambda};
\end{equation} 
the case $\phi \equiv -\pi \pmod{\pi/\lambda}$ is analogous. Substituting \eqref{case.2. real} into \eqref{4.1}, we have 
\begin{align}\label{3.1.2}
	S_{\phi,\lambda}(\alpha)&=\lambda\bigg[   \cot\left(\frac{\pi-\alpha}{2}\lambda\right)+\cot\left(\frac{3\pi+\alpha}{2}\lambda\right)\bigg]\nonumber\\
	&=\frac{2\lambda\sin\left(2\pi\lambda\right)}{\cos\left(\left(\pi+\alpha\right)\lambda\right)-\cos(2\pi\lambda)}.
\end{align}
Since $\lambda\notin \frac{1}{2}\mathbb{Z}$, the minimal positive period of $S_{\phi,\lambda}(\alpha)\not\equiv0$ in this case is $2\pi/\lambda$.

To sum up, the mimimal period of $S_{\phi,\lambda}(\alpha)$ is $2\pi/\lambda$.
\end{proof}

\begin{remark}\label{another discription of T}
Under the hypothesis $\beta\in (0,1]\setminus\left\{N^{-1}:N\in\mathbb{Z}_+\right\},$ \eqref{eq: another discription of T} in Theorem~\ref{1.1} is a consequence of \eqref{4.1}.
	
	First, \eqref{asy_0} does not equal to $0$ on
	\[
	\mathcal{D}_{t}^{O,-}:=\mathcal{D}_{t}^{O}\setminus \left(\mathcal{F}_t \cup \{(t-r,\alpha): S_{\phi,(2\beta)^{-1}}(\alpha)=0,\ 0<\alpha<2\pi\beta\}\right).
	\]
It then follows that
	\[
	\mathcal{D}_{t}^{O,-}\subset \mathrm{sing\,supp}\!\left(\frac{\sin \bigl(t\sqrt{\Delta_{\Omega}}\bigr)}{\sqrt{\Delta_{\Omega}}}\,\delta_O\right).
	\]
	
	Since $S_{\phi,(2\beta)^{-1}}(\alpha)$ is not identically zero provided $\beta\in (0,1]\setminus\left\{N^{-1}:N\in\mathbb{Z}_+\right\}$, its zero set is discrete. Since $\mathcal{F}_t$ is also a discrete set, we see that $\bar{\mathcal{D}}_{t}^{O,-}=\mathcal{D}_{t}^{O}$. As the singular support of a distribution is closed, there must hold that
	\begin{equation}\label{4.3}
		\mathcal{D}_{t}^{O} \subset \mathrm{sing\,supp}\!\left(\frac{\sin \bigl(t\sqrt{\Delta_{\Omega}}\bigr)}{\sqrt{\Delta_{\Omega}}}\,\delta_O\right).
	\end{equation}
	
	From \eqref{4.3} we deduce that
	\[
	r + r(x) \in \left\{ t : x \in \mathrm{sing\,supp}(u_\Omega(t)) \right\},
	\]
	where $u_\Omega$ is defined in \eqref{eq_1}, and $r$ and $r(x)$ denote the distance from $O$ to $P$ and $x$ to $P$ respectively. In view of \eqref{eqn: traval time}, we conclude \eqref{eq: another discription of T}.
\end{remark}
	
	\section{Recovering a corner}\label{P1.1}

Now, we are ready to prove the location and shape of a corner can be uniquely recovered from the measurement of diffracted waves in \eqref{eq_1}.

	\begin{proof}[Proof of Theorem \ref{1.1}]
		
		First, we reconstruct the location of $P$, $r$ and $r(x)$ for $x\in l$ using the travel time $t(x)$ of the diffracted wave. If there exists $P_{(1)}\neq P_{(2)}$ such that
		$$|O-P_{(1)}|+|P_{(1)}-x|=|O-P_{(2)}|+|P_{(2)}-x|.$$ 
		Then, we have 
		$$|P_{(1)}-x|-|P_{(2)}-x|=|O-P_{(2)}|-|O-P_{(1)}|=c_0,$$
		which contradicts the assumption that $l$ and $O$ do not lie on the two branches of a hyperbola, as stated in Theorem \ref{1.1}. Hence, $P\to \{t(x):x\in l\}$ is injective. Additionally, $r=|O-P|$ and $r(x)=|x-P|$ are uniquely determined.

		It remains to recover $\beta$ and $\phi$. Since $l$ is a curved line, $\pi_{[0,2\pi\beta]}(l)$ is an interval in $[0,2\pi\beta]$ as in Figure \ref{figure1}. 
		We take an open subset $l_0\subset l$ and $\pi_{[0,2\pi\beta]}(l_0)$ is an open interval in $[0,2\pi\beta]$.
		
		Find $\alpha_0\in (0,2\pi\beta)$ and $\epsilon>0$ such that $(\alpha_0-\epsilon,\alpha_0+\epsilon)\subset\pi_{[0,2\pi\beta]}(l_0)$. Since $r_{(1)}(x)=r_{(2)}(x)$ for $x\in l_0$ by $t_{(1)}(x)=t_{(2)}(x)$ for $x\in l_0$. Together with
		$$\left(u_{\mathcal{D},P,(1)}^{\mathrm{prin}}(t_{(1)}(x),x,O)\right)\bigg|_{x\in l}=\left(u_{\mathcal{D},P,(2)}^{\mathrm{prin}}(t_{(2)}(x),x,O)\right)\bigg|_{x\in l},$$ 
		we obtain
		\begin{equation}
			S_{\phi_{(1)},(2\beta_{(1)})^{-1}}(\alpha)=S_{\phi_{(2)},(2\beta_{(2)})^{-1}}(\alpha),\qquad\forall \alpha\in(\alpha_0-\epsilon,\alpha_0+\epsilon).
		\end{equation}
		Theorem \ref{T4.1} shows that $\beta_{(1)}=\beta_{(2)}$ and $\phi_{(1)}=\phi_{(2)}$. Therefore, the mapping \eqref{eqn : observation at l} is injective.
	\end{proof}

	\section{Detecting polygonal obstacles}\label{P1.2}
We are now ready to present the entire retrieval process in Theorem \ref{1.2} and verify its uniqueness. The overall detection procedure can be summarized as follows:
		\begin{enumerate}
			\item[1)] To eliminate the data pollution of reflected waves, we separate geometric and primary diffracted waves, by restricting to a suitable subsegment of $l$ with separated arrival times;
			\item[2)] We remove the measurements of secondary diffracted waves and extract primary diffracted waves, by comparing their magnitudes, as in Figure \ref{F_T_1};
			\item[3)] Through the following steps:
			\begin{enumerate}
				\item[3.1)] localization of the primary diffracted waves,
				\item[3.2)] detection of the principal terms of the localized waves,
			\end{enumerate}
			we identify the principal terms of primary diffracted waves $u_{\Omega}(t,x,O)|_{\mathcal{D}_{t,P_i}}$ passing through each $P_i$, $i=1,...,N$, which is given as in \eqref{asy_0};
			\item[4)] Applying Theorem \ref{1.1}, we then reconstruct, from remaining diffracted waves, the location and the angle of each corner in an orderly fashion.
		\end{enumerate}
		
		\begin{figure}
			\centering
			\includegraphics[width=\textwidth]{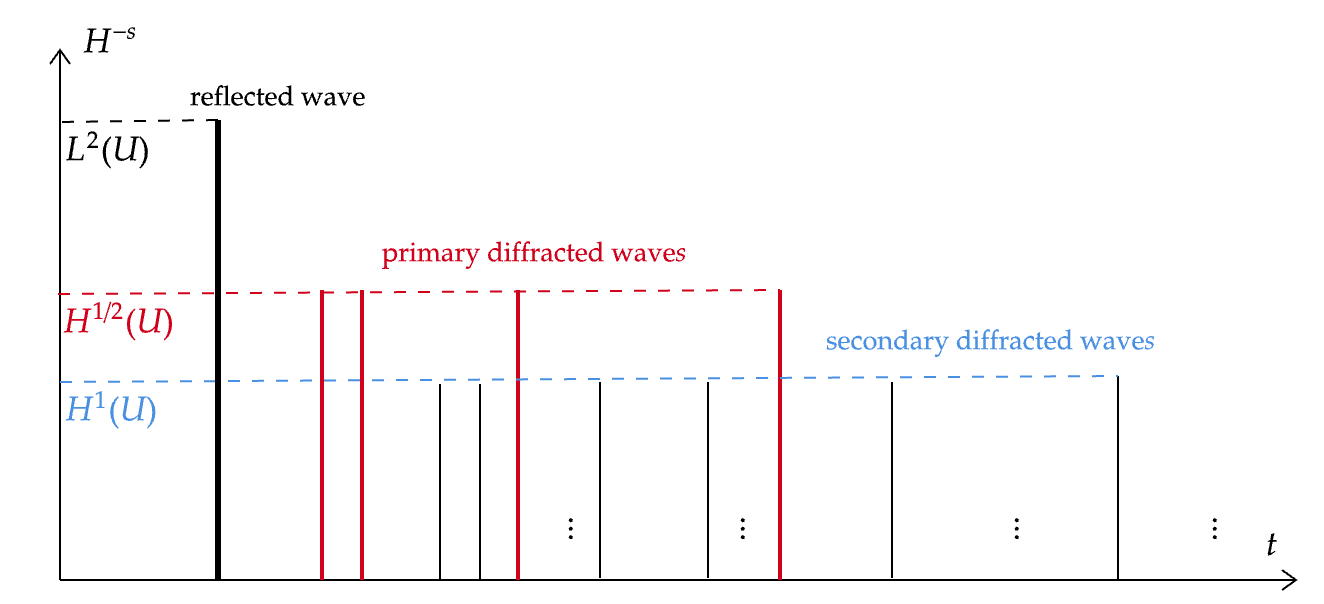}
			\caption{The Sobolev regularities of reflected and diffracted waves}
			\label{F_T_1}
		\end{figure}
		
		The strategy is to recover a polygon by applying the framework for corners, established in Theorem \ref{1.1}. We use additional shorthand notations throughout this section. Let $K$ be a convex polygon with $k$ vertices. Choose a vertex $P_0$ of $K$. Proceeding in the clockwise direction, we label the vertices by $P_i$ and the edges by $E_i$. We adopt the convention that $E_{-1} = E_{k}$. For a vertex $P_j$, let $\tilde{K}_{P_j}$ be the corner with boundary $\tilde{E}_{j-1}$, $\tilde{E}_{j+1}$, where $\tilde{E}_{j-1}$ and $\tilde{E}_{j+1}$ are the rays obtained by extending the two edges of the corner $P_j$. See e.g. Figure \ref{F1}. Then we denote the exterior domain $\tilde{\Omega}_{P_j}:=\mathbb{R}^2\setminus\tilde{K}_{P_j}$ and equip $\tilde{\Omega}_{P_j}$  with polar coordinates, taking point $P_j$ as the origin. We denote by $\{P_{v(j)}\}_{j=1,...,N}$ the collection of visible vertices of  $K$.
		
		\subsection*{Step 1) Reading diffracted waves at the receiver} Let $K$, $\Omega$, $O$ and $l$ be as in Theorem \ref{1.2}.		
It turns out that the geometry of the receiver $l$ is critical for measuring waves. Prior to initiating the recovery process, we first elaborate on the structure of $l$ and the corresponding wave data.

		\begin{lemma}\label{sep polygon} 	There exists a nonempty open curved line subsegment $l_0 \subset l$ such that the following holds.
			\begin{itemize}
				\item The primary diffracted waves arrive separately: for every $x\in \bar{l}_0$ and every pair $i\neq j$,
				\begin{equation}\label{eq : distinct arrival times}
				t_{P_{v(i)}}(x)\neq t_{P_{v(j)}}(x),\qquad
				\mbox{with $t_{P_{v(j)}}(x):=|P_{v(j)}- x| + |O - P_{v(j)}|$};
				\end{equation}
				\item The geometric waves arrive earlier than the primary diffracted waves: for every $x\in \bar{l}_0$ and $j=1,\dots,N$, 
				\begin{equation}\label{eq : geometric waves arrive early}
			\mbox{if $x\in \mathcal{G}_{t_0}^{O}$ for some $t_0>0$, then $t_0 < t_{P_{v(j)}}(x)$};\end{equation}
				\item For every $P_k\in \mathrm{Vert}(K)$, \begin{equation}\label{eqn : no touch with corner}\bar{l}_0\cap \partial \tilde{\Omega}_{P_k}=\emptyset.\end{equation}
			\end{itemize}
	
		\end{lemma}
		
		\begin{proof}
			For any $i\neq j$, the locus
			\[
			\gamma_{i,j}:=\Bigl\{ x\in \mathbb{R}^2 : |O-P_{v(i)}| + |P_{v(i)} - x| = |O-P_{v(j)}| + |P_{v(j)} - x| \Bigr\},
			\]
			 is either a (branch of a) hyperbola or a straight line.
			
			For each $j$, we write
			\[
			\mathcal{F}_{j}:=\bigcup_{t>|O-P_{v(j)}|} \bigl(\bar{\mathcal{G}}_{t}^{O}\cap \mathcal{D}_{t,P_{v(j)}}\bigr),
		\]   where \[ 
			\mathcal{D}_{t,P_{v(j)}}:=\{ x\in \Omega : |x-P_{v(j)}| = t - |O-P_{v(j)}| \}.
			\]
			By Lemma \ref{intersection}, each $\mathcal{F}_i$ is the union of at most two rays issuing from $P_{v(i)}$.
			
		Now we consider the	set
			\[
			K_e:= K \;\cup\; \left(\bigcup_{j=1}^N \mathcal{F}_j \right) \;\cup\; \left(\bigcup_{i\neq j} \gamma_{i,j}\right) \;\cup\; \left(\bigcup_{P_k\in \mathrm{Vert}(K)} \partial \tilde{\Omega}_{P_k}\right).
			\]
			This is a closed subset of $\mathbb{R}^2$, being a finite union of (segments of) lines, hyperbolae, and the polygon $K$. The assumption on $l$ in Theorem \ref{1.2} guarantees that $l$ can not be wholly contained in $K_e$, so that
			\[
			l_{\mathrm{sep}}:= l\setminus K_e \neq \emptyset.
			\]
  For any $x\in l_{\mathrm{sep}}$, we have:
			\begin{itemize}
				\item  that $x\notin \gamma_{i,j}$ means that $t_{P_{v(i)}}(x)\neq t_{P_{v(j)}}(x)$ for $i\neq j$, yielding \eqref{eq : distinct arrival times} at $x$;
				\item  that $x\notin \mathcal{F}_j$ implies that if $x\in \mathcal{G}_{t}^{O}$ for some $t > 0$ then necessarily $t < t_{P_{v(j)}}(x)$, which is exactly  \eqref{eq : geometric waves arrive early} at $x$;
				\item  that $x\notin \partial \tilde{\Omega}_{P_k}$ for all $k$ is precisely \eqref{eqn : no touch with corner} at $x$.
			\end{itemize}
			Because $K_e$ is closed, there exists an open curved line subsegment $l_0\subset l$ containing $x_0$ and disjoint from $K_e$; the three properties persist on $l_0$. This proves the existence of $l_0$ with properties \eqref{eq : distinct arrival times}–\eqref{eqn : no touch with corner}.
		\end{proof}

	Consequently, we have
	\begin{lemma}\label{lemma : nbhd of l0}
		 		There exists an open subset $U\supset \bar{l}_0$ in $\Omega$ such that, for any $\{i,j\}\subset\{1,\dots,N\}$ and $t>0$ the following holds.
		 \begin{itemize}\label{restriction of observation}
		 	\item The primary diffracted waves from distinct vertices are disjoint on $\bar{U}$:
		 	\begin{equation}\label{The primary diffracted waves have empty intersection}
		 		\mathcal{D}_{t,P_{v(i)}}\cap \mathcal{D}_{t,P_{v(j)}}\cap \bar{U}=\emptyset;
		 	\end{equation}
		 	\item The primary diffracted waves and the geometric wave are disjoint on $\bar{U}$:
		 	\begin{equation}\label{The primary diffracted waves and geometric waves have empty intersection}
		 		\bar{U}\cap \mathcal{D}_{t,P_{v(j)}}\cap \bar{\mathcal{G}}_{t}^{O} =\emptyset;
		 	\end{equation}
		 	\item No grazing rays intersect $\bar{U}$:
		 	\begin{equation}\label{The grazing ray occured by primary diffracted waves do not pass}
		 		\bar{U}\cap \bigcup_{P_k\in \mathrm{Vert}(K)} \partial \tilde{\Omega}_{P_k}=\emptyset.
		 	\end{equation}
		 \end{itemize}

	\end{lemma}	

\begin{proof}
	This lemma readily follows from Lemma \ref{sep polygon} as long as $U$ is taken to be a sufficiently small open neighbourhood of $l_0$ in $\Omega$.
\end{proof}

		\subsection*{Step 2) Extracting primary diffracted waves}

		To separate primary diffracted waves from incident and reflected waves, it is more convenient to localize the full waves within the neighbourhood $U$ given in Lemma \ref{lemma : nbhd of l0}, than to restrict the waves simply at the receiver $l$.  
		In view of \eqref{The primary diffracted waves have empty intersection}-\eqref{The grazing ray occured by primary diffracted waves do not pass}, we may choose cut-off functions\[ \chi_{0}^T, \chi_{1}^T, \cdots, \chi_{N}^T \in C_c^{\infty}(\Omega) \quad\mbox{at a fixed time $T > 0$}\] such that  
		\begin{itemize}
			\item $\mathrm{supp} \chi_\alpha^T  \cap   \mathrm{supp} \chi_\beta^T  \neq \emptyset$ for any distinct $\alpha, \beta \in \{0, 1, \cdots, N\}$;
			\item  $\chi_0^T \equiv 1$ on $\mathcal{G}_{T}^{O}\cap\bar{U}$;
			\item  $\chi_j^T \equiv 1$   on $\mathcal{D}_{T, P_{v(j)}}\cap\bar{U}$ for   $j=1,\dots,N$. 
		\end{itemize}
		Since $\{\chi_{\alpha}^T\}^N_{\alpha = 0}$ separate the geometric front and the diffracted front within $\bar{U}$, we invoke them to localize the full waves as follows
		\begin{align}\label{primary diffracted wave}
			u_{\mathcal{G}}(T,x,O)&:= \chi_0^T \, u_{\Omega}(T,x,O),\nonumber \\
			u_{\mathcal{D},P_{v(j)}}(T,x,O)&:= \chi_j^T \, u_{\Omega}(T,x,O)   \quad j = 1,\dots,N.
		\end{align}
		In fact, these localized waves serve as an approximation of the full wave. We call $u_{\mathcal{D},P_{v(j)}}(T,x,O)$ the primary diffracted wave with respect to $P_{v(j)}$.
		
		Let $u,v\in\mathcal{D}'(\Omega)$ and $H(U)$ a subspace of $\mathcal{D}'(U)$. We write $u\equiv v\mod{H(U)}$ if $(u-v)|_U\in H(U)$.
		
		\begin{proposition}\label{T_1}
			Let $K$, $O$, $\Omega$, $l$, and $\{P_{v(j)}\}_{j=1,\dots,N}$ be as in Theorem \ref{1.2}. For $U\supset \bar{l}_0$ as given in Proposition \ref{lemma : nbhd of l0}, then $u_{\Omega}$ admits the decomposition
			\begin{equation}\label{asy_full_wave}
				u_{\Omega}(T,x,O) \equiv u_{\mathcal{G}}(T,x,O) + \sum_{j=1}^{N} u_{\mathcal{D},P_{v(j)}}(T,x,O) 
				\mod \bigcap_{\epsilon>0} H^{1-\epsilon}(U),
			\end{equation}
			and this decomposition is independent of the choice of the cut-off functions. Moreover,
			\begin{equation}\label{eqn : regularity of geometric waves}
				u_{\mathcal{G}}(T,x,O)\in \bigcap_{\epsilon>0} H^{-\epsilon}(U)\setminus L^{2}(U).
			\end{equation}
		\end{proposition}
		
		\begin{remark}
			Because each of the distributions $u_{\mathcal{D},P_{v(j)}}(t,x,O)$ and $u_{\mathcal{G}}(t,x,O)$, modulo $\bigcap_{\epsilon>0} H^{1-\epsilon}(U)$, is independent of the auxiliary cut-off functions, the decomposition \eqref{asy_full_wave} implies that the component
			$$u_{\mathcal{D},P_{v(j)}}(t,x,O)\mod\bigcap_{\epsilon>0} H^{1-\epsilon}(U),\qquad t>|O-P_{v(j)}|,$$
			is the primary diffracted wave generated by the vertex $P_{v(j)}$. 
		\end{remark}
		
		\begin{proof}[Proof of \eqref{asy_full_wave}] 
			By finite speed of propagation, we can localize $u_{\Omega}(t,x,O)$ in a neighbourhood of each vertex of $K$. Since $K$ is a convex polygon, we can use the method of images to remove the edges in that neighbourhood. Consequently, the initial boundary value problem is locally equivalent to an initial value problem on a conic manifold without boundary. In this geometry, \cite{CT82b,MW04} pointed out that $u_{\Omega}(t,x,O)|_{\mathcal{D}_t^{O}\setminus\mathcal{F}_t^{O}}$ gains a $1/2-$order Sobolev regularity. The no-grazing condition at $O$ \eqref{no-grazing condition} together with \eqref{The primary diffracted waves and geometric waves have empty intersection} guarantees that $\partial K \cap \mathcal{F}_t^{O}=\emptyset$. In a sufficiently small neighbourhood $U_{\partial K}$ of $\partial K$, $u_{\Omega}(t,x,O)|_{\mathcal{D}_t^{O}}$ lies in $\bigcap_{\epsilon>0} H^{1/2-\epsilon}(U_{\partial K})$.

			The diffracted waves may diffract at other vertices, causing multiple diffractions. We first discuss the secondary diffractions. Since $u_{\Omega}(t,x,O)$ diffracts at each visible vertex $P_{v(j)}$, the diffracted part of $u_{\Omega}(t,x,O)$ propagates towards $P_i$, where $i\in\{{v(j)}-1,{v(j)}+1\}$. Define
			\begin{align*}
				\mathcal{D}_{t,P_{v(j)},P_i}:=\{x\in\Omega:\ t=|O-P_{v(j)}|+|P_{v(j)}-P_i|+|x-P_i|\},
			\end{align*}
			and
			\[
			\mathcal{F}_{t,P_{v(j)},P_i}:=\mathcal{D}_{t,P_k,P_l}\cap \bar{\mathcal{D}}_{t,P_k},\qquad 
			\mathcal{F}_{P_{v(j)},P_i}:=\bigcup_{t>0}\mathcal{F}_{t,P_{v(j)},P_i}.
			\]
			We have $\mathcal{F}_{P_{v(j)},P_i}\subset \partial\tilde{\Omega}_{P_i}$ (see Figure \ref{Ft of secondary diffracted wave}). By \eqref{The grazing ray occured by primary diffracted waves do not pass}, $u_{\Omega}(t,x,O)|_{\mathcal{D}_{t,P_{v(j)},P_i}}$ in $U$ is $1/2$ order smoother than $u_{\Omega}(t,x,O)|_{\mathcal{D}_{t,P_{v(j)}}}$, i.e.
			\begin{equation}\label{eqn : regularity of diffracted wave}
				u_{\Omega}(t,x,O)|_{\mathcal{D}_{t,P_{v(j)},P_i}}\in \bigcap_{\epsilon>0} H^{1-\epsilon}(U),
			\end{equation}
			which means secondary diffracted waves are $1/2$-order smoother than primary diffracted waves. By energy estimates for the wave equation, multiple diffracted waves are never stronger than secondary diffracted waves, and also lie in space $\bigcap_{\epsilon>0} H^{1-\epsilon}(U).$

			Consequently,  we may consider the quotient mapping
			\begin{equation}\label{quotient mapping}
				[\cdot]_{H^{-0}/ H^{1-0}(U)}: \bigcap_{\varepsilon>0} H^{-\varepsilon}(U)\longrightarrow \bigcap_{\varepsilon>0} H^{-\varepsilon}(U)/\bigcap_{\varepsilon>0} H^{1-\varepsilon}(U).
			\end{equation}
			By propagation of singularities on conic manifolds (Theorem \ref{conic}), the remaining terms of $[u_{\Omega}(T,x,O)]_{H^{-0}/ H^{1-0}(U)}$ on $U$ arise from:
			\begin{enumerate}
				\item[(i)] the geometric (incident/reflected) waves supported on $\mathcal{G}_T^{O}$, and
				\item[(ii)] the family of primary diffracted waves supported on the sets $\mathcal{D}_{T,P_{v(j)}}$.
			\end{enumerate}
			Therefore, this establishes the decomposition \eqref{asy_full_wave} and shows that it is independent of the choice of cut-off functions.
			\end{proof}

		\begin{figure}[h]
			\centering
			\includegraphics[width=\textwidth]{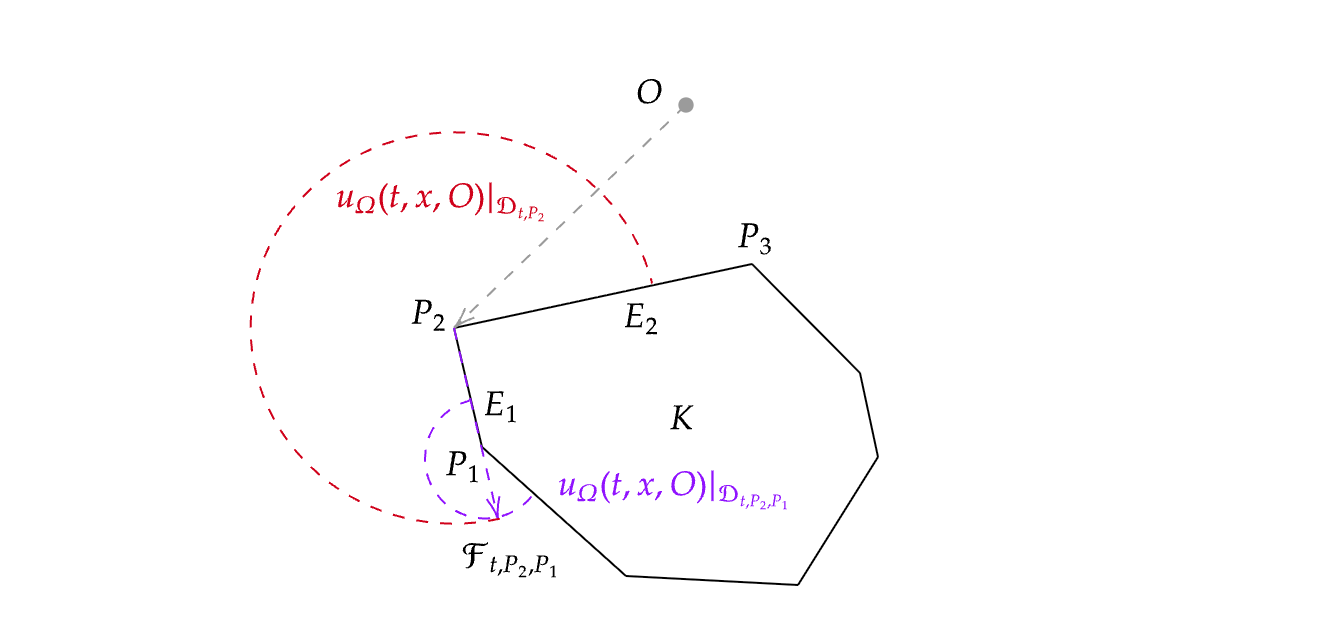}
			\caption{$\mathcal{F}_{t,P_2,P_1}$ of the secondary diffracted wave  $u_{\Omega}(t,x,O)|_{\mathcal{D}_{t,P_2,P_1}}$.}
			\label{Ft of secondary diffracted wave}
		\end{figure}

			\begin{proof}[Proof of \eqref{eqn : regularity of geometric waves}]

			  By Sobolev embedding theorem, $$\delta_O \in \bigcap_{\epsilon>0}H^{-1-\epsilon}(\Omega)\setminus H^{-1}(\Omega).$$

			  For an incident wave $u_{\mathcal{G}}^{\mathrm{inc}}$ travelling from $O$ to $x$, the classical propagation of Sobolev singularities along bicharacteristics (see \cite[Theorem 6.1.1']{Ho_FIO_2}) readily shows that  $$u_{\mathcal{G}}^{\mathrm{inc}}(T,x,O)\in \bigcap_{\epsilon>0}H^{-\epsilon}(U)\setminus L^2(U).$$
			
			For a reflected wave $u_{\mathcal{G}}^{\mathrm{ref}}$ hitting the boundary at $t_0$ in a point $x_0 \in \partial K$, we localize $u_{\mathcal{G}}^{\mathrm{ref}}$ in an open neighbourhood $U_{x_0}\subset \mathbb{R}^2$ near the reflection point $x_0$. Since $\partial K\cap U_{x_0}$ is a line segment, we apply the method of images to \eqref{eq_1} within $U_{x_0}$. Consequently, the local boundary initial value problem reduces to an initial value problem. The source for reflected waves lies in $\bigcap_{\epsilon>0}H^{-\epsilon}(U_{x_0})\setminus L^2(U_{x_0}).$ By propagation of Sobolev singularities, $u_{\mathcal{G}}^{\mathrm{ref}}$ at time $T$ still satisfies \eqref{eqn : regularity of geometric waves}, which completes the proof.
		\end{proof}

		\subsection*{Step 3) Identifying principal terms of diffracted waves}
		\begin{proposition}\label{T_2}
			Let $U$, $T$ and $u_{\mathcal{D},P_{v(j)}}(t,x,O)$ be as in Proposition \ref{T_1}, where $j\in\{1,\dots,N\}$. We have
			\begin{equation}\label{eqn : poly pricipal term}
				u_{\mathcal{D},P_{v(j)}}(T,x,O) \equiv u_{\mathcal{D},P_{v(j)}}^{\mathrm{prin}}(T,x,O)\mod{\bigcap_{\epsilon>0} H^{1-\epsilon}(U)},
			\end{equation}
			where $u_{\mathcal{D},P_{v(j)}}^{\mathrm{prin}}(t,x,O)$ is given by \eqref{asy_0} with $C(I)=\tilde{\Omega}_{P_{v(j)}}$.
		\end{proposition}
		
		\begin{proof}
		We prove \eqref{eqn : poly pricipal term} in the following steps.
			
			\subsubsection*{Step 3.1) Localization of the primary diffracted wave}
			First, we microlocalize the source such that the incident wave only hits the boundary $\partial K$ in a sufficiently small neighborhood of the vertex $P_{v(j)}$. To achieve this, we may choose a pseudo-differetial operator  $$\chi_{P_{v(j)}}(x,D)\in \Psi^0(\Omega)$$ satisfying:
			\begin{itemize}
				\item $\chi_{P_{v(j)}}(x,\xi)$ is homogeneous of degree $0$ and real-valued;
				\item $\chi_{P_{v(j)}}(x,\xi)=1$ near $(O,\xi_{P_{v(j)}}^O)$, where $\xi_{P_{v(j)}}^O\in T^*_O\mathbb{R}^2$ is the covector pointing from $O$ to $P_{v(j)}$;
				\item $\mathrm{WF}(\chi_{P_{v(j)}}(x,D))$ is contained in a small conic neighbourhood of $(O,\xi_{P_{v(j)}}^O)$.
			\end{itemize}
			Since $\delta_O \in \mathcal{E}'(\Omega)$, it follows that $\chi_{P_{v(j)}}(x,D)\delta_O \in \mathcal{E}'(\Omega)$. 
			
			Denote by $\Delta_\Omega$ and $\Delta_j$ the Neumann Laplacians on $\Omega$ and on $\tilde{\Omega}_{P_{v(j)}}$, respectively. Let $W_{\Omega}(s)$ and $W_{\tilde{\Omega}_{P_{v(j)}}}(s)$ be the wave kernels with respect to $\Delta_\Omega$ and $\Delta_j$, respectively. The waves with Cauchy data 
			\[
			\left(v(0,x),\,\partial_t v(0,x)\right)\in \mathcal{D}'(\Omega)\times \mathcal{D}'(\Omega),
			\]
			take the form:
			\begin{align}\label{solution operater}
				W_{\Omega}(s)v(0,x)
				&:=\left(\cos \left(s\sqrt{\Delta_\Omega}\right)\right)v(0,x)
				\;+\;
				\left(\frac{\sin \left(s\sqrt{\Delta_{\Omega}}\right)}{\sqrt{\Delta_{\Omega}}}\right)\,
				\partial_t v(0,x),\\
				W_{\tilde{\Omega}_{P_{v(j)}}}(s)v(0,x)
				&:=\left(\cos \left(s\sqrt{\Delta_j}\right)\right)v(0,x)
				\;+\;
				\left(\frac{\sin \left(s\sqrt{\Delta_{j}}\right)}{\sqrt{\Delta_{j}}}\right)\,
				\partial_t v(0,x).
			\end{align}
			
			Consider the microlocalized wave
			\begin{equation}\label{cut-off wave on Omega}
				u_{\Omega,\chi_{P_{v(j)}}}(t,x,O):=W_{\Omega}(s)\left(\chi_{P_{v(j)}}(x,D)\delta_O\right).
			\end{equation}
			Figure \ref{F1} illustrates the microlocalization of the incident wave.
			\begin{figure}[h]
				\centering
				\includegraphics[width=\textwidth]{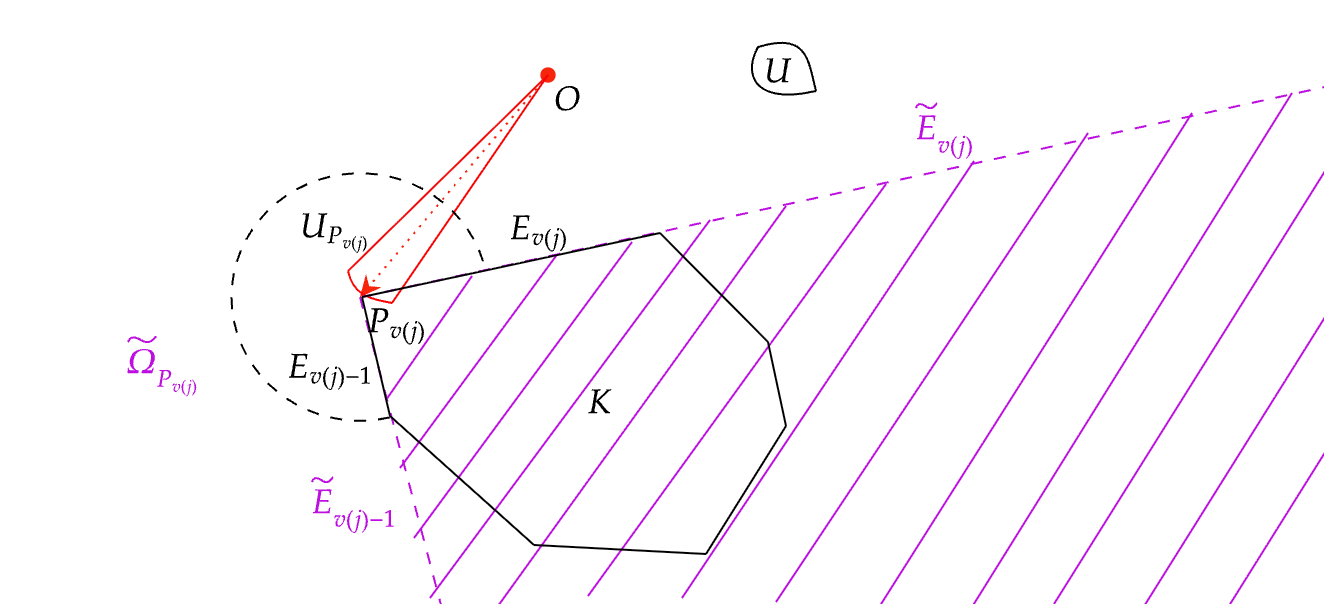}
				\caption{Microlocalization of the incident wave.}
				\label{F1}
			\end{figure}

			By the propagation of singularities near the vertices \eqref{conic propagation conical} in Theorem \ref{conic}, there exists $\epsilon>0$ such that
			\[
			u_{\Omega}(t,x,O) - u_{\Omega,\chi_{P_{v(j)}}}(t,x,O)\in \mathcal{D}^\infty_{\mathrm{loc}}(U_{P_{v(j)}}),
			\]
			for $t\in (r_{v(j)},r_{v(j)}+\epsilon]$. Since $\mathrm{sing\, supp}(u_{\mathcal{G}}(t))\subset \mathcal{G}_t^{O}$ and $\mathrm{supp}\chi_{j}^T\cap\mathcal{G}_t^{O}=\emptyset$,  \eqref{asy_full_wave} implies
			\begin{equation}\label{C1.0}
				u_{\mathcal{D},P_{v(j)}}(T,x,O) \equiv \chi_{j}^Tu_{\Omega,\chi_{P_{v(j)}}}(T,x,O)\pmod{C^\infty(U)}.
			\end{equation}
			
			Wave \eqref{cut-off wave on Omega} in $\Omega$ emanates from $O$ along the directions in a neighbourhood of $\xi_{P_{v(j)}}$. The wave front of \eqref{cut-off wave on Omega} first hits $P_{v(j)}$ at time $r_{v(j)}$. Then, the diffraction occurs. The diffracted wave propagates towards $P_{v(j)-1}$ and $P_{v(j)+1}$ causing secondary diffracted waves at time $r_{v(j)}+|E_{v(j)-1}|$ and $r_{v(j)}+|E_{v(j)}|$, respectively.
			
			We consider the microlocalized wave 
			\begin{equation}\label{cut-off wave on model cone}
				u_{\tilde{\Omega}_{P_{v(j)}},\chi}(t,x,O):=W_{\tilde{\Omega}_{P_{v(j)}}}(s)\left(\chi_{P_{v(j)}}(x,D)\delta_O\right)
			\end{equation}
			on $\tilde{\Omega}_{P_{v(j)}}$. Wave \eqref{cut-off wave on model cone} also propagates from $O$ along the directions around $\xi_{P_{v(j)}}$ and strikes $P_{v(j)}$ at time $r_{v(j)}$. However, no secondary diffraction takes place afterwards, since there is only one vertex in $\tilde{\Omega}_{P_{v(j)}}$. 
			
			We claim that \eqref{cut-off wave on Omega} coincides with \eqref{cut-off wave on model cone}, up to a smooth function, before the secondary diffraction. Namely, for $t\in [0, r_{v(j)}+L)$,
				\begin{equation}\label{C1.1}
					u_{\Omega,\chi_{P_{v(j)}}}(t,x,O)\equiv u_{\tilde{\Omega}_{P_{v(j)}},\chi}(t,x,O)\pmod{C^\infty(\tilde{\Omega}_{P_{v(j)}})}.
				\end{equation}
				where $L:=\min\{|E_{{v(j)}-1}|,|E_{{v(j)}}|\}$. 
			
			\begin{proof}[Proof of \eqref{C1.1}]
				Since $\chi_{P_{v(j)}}(x,D)\delta_O$ is compactly supported in $\Omega$, by finite speed of propagation, the supports of \eqref{cut-off wave on Omega} and \eqref{cut-off wave on model cone} agree before the waves reach $\partial K$. Hence, there exists $\epsilon_1>0$ such that \eqref{C1.1} holds for $t\in [0,\epsilon_1)$. Let
				\[
				T_*:=\sup\{\,T\in [0,r_{v(j)}+L] : \eqref{C1.1} \text{ holds for }t\in [0,T)\,\}.
				\]
				We prove $T_*= r_{v(j)}+L$ by a contradiction argument. Assume  $T_*< r_{v(j)}+L$ and then there exists $\epsilon>0$ so that for all $$t\in (T_*-\epsilon,T_*+\epsilon)\subset[0,r_{v(j)}+L].$$
				
				Denote the singular support of \eqref{cut-off wave on Omega} 
				\[
				V_{\mathrm{sing}}^t:=\mathrm{sing\,supp}\left(u_{\Omega,\chi_{P_{v(j)}}}(t,\cdot,O)\right),\quad\mbox{for $t\in (T_*-\epsilon,T_*+\epsilon)$}.
				\] 
				Choose an open set $V_{\max}\subset \bar{\Omega}$ such that
				\begin{itemize}
					\item $V_{\mathrm{sing}}^{t}\subset V_{\max}\subset \bar{\tilde{\Omega}}_{P_{v(j)}}$;
					\item $\{x\in\Omega:d(x,V_{\max})\leq s_0\}\subset \bar{\tilde{\Omega}}_{P_{v(j)}}$ for some $s_0>0$;
					\item $\mathrm{Vert}(K)\cap V_{\max}=\{P_j\}$.
				\end{itemize}
				Then, we pick $\chi\in C_c^\infty(\Omega)$ such that
				\begin{itemize}
					\item $\mathrm{supp}(\chi)\subset V_{\max}$;
					\item $\chi\equiv1$ on $V_{\mathrm{sing}}^t$.
				\end{itemize}
				
				Since \eqref{C1.1} holds on $[0,T_*)$, the two waves agree modulo $C^\infty(\tilde{\Omega}_{P_{v(j)}})$ at  time $t_*\in (T_*-\epsilon/2,T_*)$. For any $s\in (0,s_0)$ and $x\in V_{\mathrm{sing}}^{t_*+s}$, by the finite speed of propagation, 
				\begin{align*}
					u_{\Omega,\chi_{P_{v(j)}}}(t_*+s,x,O)
					&= W_{\Omega}(s)\left(\chi u_{\Omega,\chi_{P_{v(j)}}}(t_*,\cdot,O)\right)(x)\\
					&= W_{\tilde{\Omega}_{P_{v(j)}}}(s)\left(\chi u_{\Omega,\chi_{P_{v(j)}}}(t_*,\cdot,O)\right)(x).
				\end{align*}
				Replacing $u_{\Omega,\chi_{P_{v(j)}}}$ by $u_{\tilde{\Omega}_{P_{v(j)}},\chi}$ on the right gives
				\begin{equation*}
					u_{\Omega,\chi_{P_{v(j)}}}(t_*+s,x,O) \equiv u_{\tilde{\Omega}_{P_{v(j)}},\chi}(t_*+s,x,O)\pmod{C^\infty(V_{\mathrm{sing}}^{t_*+s})}.
				\end{equation*}
				This proves \eqref{C1.1} holds for $t\in [0,T_*+s)$. This contradicts the definition of $T_*$. Therefore $T_*= r_{v(j)}+L$, proving \eqref{C1.1}.
			\end{proof}

			\subsubsection*{Step 3.2) Reading the principal terms of localized waves}

			An analogous arguement with \eqref{C1.0} yields that
			\begin{equation}\label{C1.2}
				u_{\tilde{\Omega}_{P_{v(j)}}}(t,x,O)-u_{\tilde{\Omega}_{P_{v(j)}},\chi}(t,x,O)
				\in \mathcal{D}^\infty_{\mathrm{loc}}(U_{P_{v(j)}}), \qquad
				t\in (r_{v(j)},\, r_{v(j)}+\epsilon_2].
			\end{equation}
		    By \eqref{principal diffracted}, for any open set $U_0$ with $\bar{U}_0\cap \mathcal{F}_t^{O}=\emptyset$,
			\begin{equation}\label{C1.3}
				\chi_{j}^Tu_{\tilde{\Omega}_{P_{v(j)}}}(t,x,O)-u_{\mathcal{D},P_{v(j)}}^{\mathrm{prin}}(t,x,O)\in \bigcap_{\epsilon>0}H^{1-\epsilon}(U_0),
			\end{equation}
			where $u^{\mathrm{prin}}_{\mathcal{D},P_{v(j)}}(t,x,O)$ is given in \eqref{asy_0}.
			
			Combining \eqref{C1.0}, \eqref{C1.1}, \eqref{C1.2}, and \eqref{C1.3}, we obtain that there exists $\epsilon_3>0$ such that for $t_0\in (r_{v(j)},r_{v(j)}+\epsilon_3)$, with $s:=T-t_0$,
			\begin{equation}\label{C1.renormalize}
				u_{\mathcal{D},P_j}(T,x,O)
				\equiv
				W_{\Omega}(s)\left(u_{\mathcal{D},P_{v(j)}}^{\mathrm{prin}}(t_0,\cdot,O)\right)(x)
				\mod{\bigcap_{\epsilon>0} H^{1-\epsilon}(U)}.
			\end{equation}
			
			The proof of \eqref{eqn : poly pricipal term} is concluded by \eqref{C1.renormalize} and the following identity
			\begin{equation}\label{C1.4}
				W_{\Omega}(s)\left(u_{\mathcal{D},P_{v(j)}}^{\mathrm{prin}}(t_0,\cdot,O)\right)(x)
				\equiv u_{\mathcal{D},P_{v(j)}}^{\mathrm{prin}}(T,x,O)
				\mod{\bigcap_{\epsilon>0} H^{1-\epsilon}(U)}.
			\end{equation}
		\end{proof}
		
			It remains to verify \eqref{C1.4}.

			\begin{proof}[Proof of \eqref{C1.4}]

				By the method of images, the initial boundary value problem \eqref{eq_1} for $\tilde{\Omega}_{P_{v(j)}}$ is locally equivalent to an initial value problem on conic manifold without boundary. The microlocal structure of the wave kernel \eqref{conormal order} in Proposition \ref{conormal diffracted wave} gives
				\begin{equation*}
					u_{\mathcal{D},P_{v(j)}}^{\mathrm{prin}}(t_0,\cdot,O)
					=\left(\frac{\sin \left(t_0\sqrt{\Delta_{j}}\right)}{\sqrt{\Delta_{j}}}\right)\left(\chi_{P_{v(j)}}(x,D)\delta_O\right)
					\in I^{-1}(\Omega,N^*(\mathcal{D}_{t_0,P_{v(j)}})).
				\end{equation*}
				
				To filter out the secondary and multiple diffracted wave, we introduce a cut-off function $\tilde{\chi}_{P_{v(j)}}(x)\in C^\infty_c(\tilde{\Omega}_{P_{v(j)}})$ satisfying 
				$$\tilde{\chi}_{P_{v(j)}}(x)=1\mbox{ near } \mathcal{D}_{t_0,P_{v(j)}}\cap\left\{x\in \tilde{\Omega}_{P_{v(j)}}: x-t\xi_{P_{v(j)}}^x\in U\mbox{ for some }t>0\right\}.$$ 
				The cut-off wave $\tilde{\chi}_{P_{v(j)}}u_{\mathcal{D},P_{v(j)}}^{\mathrm{prin}}(t_0,x,O)$ is compactly supported in $\tilde{\Omega}_{P_{v(j)}}$, and thus does not cause secondary diffraction at other vertices in $\Omega$. By \eqref{conic propagation inner} in Theorem \ref{conic}, together with \eqref{eqn : regularity of diffracted wave}, we obtain
				\begin{equation}\label{cut-off primary diffrated waves}
					W_{\Omega}(s)u_{\mathcal{D},P_{v(j)}}^{\mathrm{prin}}(t_0,x,O)\equiv 
					W_{\Omega}(s)\left(\tilde{\chi}_{P_{v(j)}}u_{\mathcal{D},P_{v(j)}}^{\mathrm{prin}}(t_0,\cdot,O)\right)(x)\mod{\bigcap_{\epsilon>0}H^{1-\epsilon}(U)}.
				\end{equation}
				
				Applying the standard FIO parametrix construction for hyperbolic equations (for example \cite[Theorem 5.1.2]{Du_FIO}) gives
				\[
				W_{\Omega}(s)\left(\tilde{\chi}_{P_{v(j)}}u_{\mathcal{D},P_{v(j)}}^{\mathrm{prin}}(t_0,\cdot,O)\right)(x)\in I^{-1}\left(U,N^*(U\cap\mathcal{D}_{s+t_0,P_{v(j)}})\right).
				\]
				Moreover, by \eqref{pricipal symbol of conomal}, its principal symbol lies in
				\[
				\sigma_{p}\left(W_{\Omega}(s)\left(\tilde{\chi}_{P_{v(j)}}u_{\mathcal{D},P_{v(j)}}^{\mathrm{prin}}(t_0)\right)\right)\in
				S^{-1}\left(N^*(U\cap \mathcal{D}_{s+t_0,P_{v(j)}})\right).
				\]
				By the transport equation for principal symbols in \cite[Theorem 5.1.2]{Du_FIO} and the principal amplitude of the wave kernel \eqref{principal symbol} in Proposition \ref{conormal diffracted wave}, for $x\in U$,
				\begin{align*}
					\lefteqn{\sigma_{p}\left(W_{\Omega}(s)\left(\tilde{\chi}_{P_{v(j)}}u_{\mathcal{D},P_{v(j)}}^{\mathrm{prin}}(t_0)\right)\right)(x,-\xi_{P_{v(j)}}^x)}\\
					&=\sigma_{p}\left(u_{\mathcal{D},P_{v(j)}}^{\mathrm{prin}}(t_0)\right)(x-s\xi_{P_{v(j)}}^x,-\xi_{P_{v(j)}}^x) \\
					&=\sigma_{p}\left(u_{\mathcal{D},P_{v(j)}}^{\mathrm{prin}}(s+t_0)\right)(x,-\xi_{P_{v(j)}}^x)\\
					&=\sigma_{p}\left(u_{\mathcal{D},P_{v(j)}}^{\mathrm{prin}}(T)\right)(x,-\xi_{P_{v(j)}}^x).
				\end{align*}
				By \eqref{leading order singularity} and \eqref{principal symbol}, the quantization of principal symbols gives
				\begin{multline*}
					W_{\Omega}(s)\left(\tilde{\chi}_{P_{v(j)}}u_{\mathcal{D},P_{v(j)}}^{\mathrm{prin}}(t_0,\cdot,O)\right)(x)- u_{\mathcal{D},P_{v(j)}}^{\mathrm{prin}}(T,x,O)\\
					\in I^{-3/2}\left(U,N^*(U\cap\mathcal{D}_{s+t_0,P_{v(j)}})\right).
				\end{multline*}
				In terms of Sobolev regularities of conormal distributions, this amounts to
					\begin{equation}\label{primary term of cut-off primary diffracted waves}
					W_{\Omega}(s)\left(\tilde{\chi}_{P_{v(j)}}u_{\mathcal{D},P_{v(j)}}^{\mathrm{prin}}(t_0,\cdot,O)\right)(x)- u_{\mathcal{D},P_{v(j)}}^{\mathrm{prin}}(T,x,O)\\
					\in \bigcap_{\epsilon>0}H^{1-\epsilon}(U).
				\end{equation}
				
				In summary, \eqref{cut-off primary diffrated waves} and \eqref{primary term of cut-off primary diffracted waves} prove \eqref{C1.4}. 
			\end{proof}

		\subsection*{Step 4) Retrieving the obstacle}
		Now, we are ready to complete the detection of the polygonal obstacle.
		\begin{proof}[Proof of Theorem \ref{1.2}]
			By Lemma \ref{sep polygon} and Lemma \ref{lemma : nbhd of l0}, choose curved line subsegment $l_0\subset l$ and $U\supset \bar{l}_0$ on which one can capture a primary diffracted front among all diffracted and geometric fronts. 
			
			In order to extract the primary diffracted waves in \eqref{primary diffracted wave}, we consider the quotient mapping \eqref{quotient mapping} the information of which is fully determined by \eqref{eqn : poly observation at l}. By \eqref{eqn : regularity of diffracted wave}, secondary and multiple diffractions are removed in the quotient space 
			$$\bigcap_{\epsilon>0} H^{-\epsilon}(U) /\bigcap_{\epsilon>0} H^{1-\epsilon}(U).$$
			In the meanwhile, by \eqref{eqn : regularity of geometric waves}, we can identify all reflected waves in the quotient space 
			$$\bigcap_{\epsilon>0} H^{-\epsilon}(U) /\bigcap_{\epsilon>0} H^{1/2-\epsilon}(U).$$
			By Proposition~\ref{T_2}, the primary diffracted waves in \eqref{primary diffracted wave} coincide with the principal diffracted terms \eqref{asy_0} modulo $\bigcap_{\epsilon>0} H^{1-\epsilon}(U)$.
			
			For each visible vertex $P_{v(j)}$, let $t_{P_{v(j)}}(x)=|O-P_{v(j)}|+|x-P_{v(j)}|$ be the arrival time of the primary diffracted wave with respect to $P_{v(j)}$ at $x$. If $t_{P_{v(j)},(1)}=t_{P_{v(j)},(2)}$ on $l_0$, we use Theorem \ref{1.1} to obtain $P_{{v(j)},(1)}=P_{{v(j)},(2)}$ and $r_{{v(j)},(1)}=r_{{v(j)},(2)}$.

			Let $\left\{u_{\Omega,(i)}\right\}_{i=1,2}$ be the equation of \eqref{eq_1} with respect to convex polygons $\{K_{(i)}\}_{i=1,2}.$ 
			For $x_0\in l$, we choose $\{U_{(i)}\}_{i=1,2}$ by Proposition \ref{lemma : nbhd of l0}. If for all $j\in \mathbb{Z}_+$,
			$$[u_{\Omega,(1)}(t_j(x_0),x,O)]_{x_0}=[u_{\Omega,(2)}(t_j(x_0),x,O)]_{x_0},$$ 
			then there exists an open set $U\ni x_0$ such that  $$u_{\Omega,(1)}(t_{P_{v(j)}}(x_0),x,O)=u_{\Omega,(2)}(t_{P_{v(j)}}(x_0),x,O),\quad\mbox{for $x\in U$ and $j=1,...,N$}.$$ 
			Denote $\tilde{U}=U_{(1)}\cap U_{(2)}\cap U$.  It follows that
			$$\left[u_{\Omega,(1)}(t_{P_{v(j)}}(x_0),x,O)\right]_{H^{-0}/ H^{1-0}(\tilde{U})}=\left[u_{\Omega,(2)}(t_{P_{v(j)}}(x_0),x,O)\right]_{H^{-0}/ H^{1-0}(\tilde{U})}.$$ 
			Consequently,  for each $j=1,\dots,N$ and $x_0\in l_0$,
			\begin{equation}\label{eqn : principal term equiv}
				u_{\mathcal{D},P_{v(j)},(1)}^{\mathrm{prin}}(t_{P_{v(j)}}(x_0),x,O)\equiv u_{\mathcal{D},P_{v(j)},(2)}^{\mathrm{prin}}(t_{P_{v(j)}}(x_0),x,O)\mod{ \bigcap_{\epsilon>0} H^{1-\epsilon}(\tilde{U})}.
			\end{equation}
			
			By \eqref{asy_0}, equality \eqref{eqn : principal term equiv} amounts to
			\begin{multline*}
				H(t-r-r(x))\bigg[\sin\left(\pi\sqrt{\Delta_{I_j^{(1)}}}\right)\delta_{\phi_{v(j),(1)}}+\sin\left(\pi\sqrt{\Delta_{I_j^{(1)}}}\right)\delta_{-\phi_{v(j),(1)}}\\
				-\sin\left(\pi\sqrt{\Delta_{I_j^{(2)}}}\right)\delta_{\phi_{v(j),(2)}}-\sin\left(\pi\sqrt{\Delta_{I_j^{(2)}}}\right)\delta_{-\phi_{v(j),(2)}}\bigg]\\ \equiv 0, \mod{ \bigcap_{\epsilon>0} H^{1-\epsilon}(\tilde{U})}.
			\end{multline*}
			Since $H(t-r-r(x))\in \bigcap_{\epsilon>0} H^{1/2-\epsilon}(\tilde{U})\setminus H^{1/2}(\tilde{U})$, we obtain that there exist $\alpha_0$ and $\epsilon>0$ such that
			\[
			S_{\phi_{v(j),(1)},(2\beta_{v(j),(1)})^{-1}}(\alpha)
			=
			S_{\phi_{v(j),(2)},(2\beta_{v(j),(2)})^{-1}}(\alpha),
			\qquad \forall \alpha\in (\alpha_0-\epsilon,\alpha_0+\epsilon).
			\]
			where $S_{\phi,\lambda}$ is given by \eqref{4.1}. Applying Theorem \ref{T4.1}, we conclude
			\begin{multline*}
				\bigl\{ (P_{{v(j)},(1)}, r_{{v(j)},(1)}, \phi_{v(j),(1)}, \beta_{v(j),(1)}) : j=1,\dots,N_{(1)} \bigr\}
				=\\
				\bigl\{ (P_{{v(j)},(2)}, r_{v(j),(2)}, \phi_{v(j),(2)}, \beta_{v(j),(2)}) : j=1,\dots,N_{(2)} \bigr\},
			\end{multline*}
			which completes the proof.
		\end{proof}

	  	\bigskip
	
	\noindent {\bf Acknowledgements.} The authors were supported in part by NSFC.  
	Views and opinions expressed are those of the authors only and do not necessarily reflect those of the funding organizations.

	\bigskip	\noindent {\bf Data Availability Statement.} Data sharing not applicable to this article as no datasets were generated or analysed during the current study.
	
	\bigskip	\noindent {\bf Conflict of Interest.} The authors have no conflicts of interest to declare that are relevant to the content of this article.

	\bibliography{ref.bib}
	\bibliographystyle{amsplain}
	
\end{document}